\documentclass[hidelinks,onefignum,onetabnum]{siamart251216}

\usepackage[T1]{fontenc}
\usepackage[utf8]{inputenc}

\usepackage{graphicx}
\usepackage{amssymb}
\usepackage{latexsym}
\usepackage{setspace}
\usepackage{amsmath}
\usepackage{mathrsfs}
\usepackage{amsfonts}
\usepackage{amsxtra}
\usepackage{stmaryrd}
\usepackage{mathtools}
\usepackage{esint}

\usepackage{booktabs}
\usepackage{siunitx}
\usepackage{multirow}
\usepackage{xcolor}
\usepackage{url}
\usepackage{subcaption}
\newcommand{\new}[1]{{\color{blue}#1}}

\newtheorem{remark}[theorem]{Remark}

\makeatletter
\newcommand{\chapterauthor}[1]{%
    {\parindent0pt\vspace*{-25pt}%
    \linespread{1.1}\large\scshape#1%
    \par\nobreak\vspace*{35pt}}
    \@afterheading%
}
\makeatother

\newcommand{\yedit}[1]{#1}

\title{Structure-Preserving Dynamical Low-Rank Approximations
for Stochastic Vlasov--Poisson Equations}

\author{
Jianbo Cui\thanks{
Department of Applied Mathematics, The Hong Kong Polytechnic University,
Hung Hom, Kowloon, Hong Kong, SAR, China.
Email: \texttt{jianbo.cui@polyu.edu.hk}
}
\and
Carmela Scalone\thanks{
Dipartimento di Ingegneria e Scienze dell'Informazione e Matematica,
Universit\`a dell'Aquila.
Email: \texttt{carmela.scalone@univaq.it}
}
}

\date{}

\begin{document}
\maketitle

\begin{abstract}
{We develop structure-preserving dynamical low-rank methods for stochastic Vlasov--Poisson equations with transport noise. We first derive continuous low-rank evolution equations that inherit the mass, momentum, and energy balance laws through fixed velocity modes. We then construct two augmented basis-update Galerkin (BUG) integrators based on Euler--Maruyama discretization of the It\^o formulation and Heun-type discretization of the Stratonovich formulation. Both integrators combine standard basis augmentation with conservative rank truncation to preserve the selected velocity moments. We prove pathwise mass conservation and conservation of total momentum in expectation. The energy balances contain distinct stochastic corrections with identical conditional expectations and an explicitly characterized time-discretization defect. Numerical experiments illustrate the conservation properties, rank adaptation, and temporal convergence behavior of the proposed methods.}

\end{abstract}

\section{Introduction}

The Vlasov--Poisson equation is a fundamental kinetic model describing the evolution of charged particles under the action of a self-consistent electric field.

Random external fields provide models for fluctuating forces acting on charged particles, with motivations from turbulent transport in fusion and astrophysical plasmas; see \cite{BP2022, GIVW} and the references therein. In a related weak-turbulence setting, \cite{BB} derived velocity-space diffusion from particle dynamics driven by a prescribed random electric field. 
In this work, we consider a stochastic Vlasov--Poisson equation with transport noise in the velocity variable, written in Stratonovich form as
\begin{align}\label{sto-Vlasov}
df(t,x,v)
&+
\left(
v\cdot\nabla_x f(t,x,v)
+
E(t,x)\cdot\nabla_v f(t,x,v)
\right)dt\\ \nonumber
&+
\sum_{k=1}^{K}
\sigma_k(x)\cdot\nabla_v f(t,x,v)\star d\beta_k(t)
=0 ,
\end{align}
where $x\in\mathbb T^{\text d}$, $v\in\mathbb R^{\text d}$, and $\{\beta_k\}_{k=1}^{K}$ are independent Brownian motions on a complete filtered probability space $(\Omega,\mathcal F,\{\mathcal F_t\}_{t\ge 0},\mathbb P)$.
The symbol $\star$ denotes the Stratonovich integral, and the coefficients $\sigma_k:\mathbb T^{\text d}\to\mathbb R^{\text d}$ describe the stochastic transport. 
The electric field is determined self-consistently through the Poisson equation
\begin{align*}
E(t,x)=-\nabla u(t,x),\qquad
-\Delta u(t,x)=\int_{\mathbb R^{\text d}}f(t,x,v)\,dv-1 .
\end{align*}
Equation \eqref{sto-Vlasov} admits the equivalent It\^o formulation
\begin{align}\label{sto-Vlasov-ito}
df & +(v\cdot \nabla_x f+E(t,x)\cdot \nabla_v f)dt+\sum_{k=1}^{K} \sigma_k(x) \cdot \nabla_v f d\beta_k(t)  \\ \nonumber
& =\frac 12\sum_{i,j=1}^{{\text d}} (\sigma \sigma^\top)_{i,j}(x)\partial_{v_iv_j}f dt,
\end{align}
where the matrix  $\sigma:=(\sigma_1,\cdots,\sigma_{K}) \in \mathbb R^{\text d\times K}.$

Let
\[
\rho(t,x)=\int_{\mathbb R^{\text d}} f(t,x,v)\,dv,
\qquad
J(t,x)=\int_{\mathbb R^{\text d}} v f(t,x,v)\,dv .
\]
The total mass and total momentum are
\[
\mathcal M[f(t)]=\int_{\mathbb T^{\text d}}\rho(t,x)\,dx,
\qquad
\mathcal P[f(t)]=\int_{\mathbb T^{\text d}}J(t,x)\,dx,
\]
and the Hamiltonian is
\[
\mathcal H[f(t)]
=\frac12\int_{\mathbb T^{\text d}\times\mathbb R^{\text d}}|v|^2f(t,x,v)\,dv\,dx
+\frac12\int_{\mathbb T^{\text d}}|E(t,x)|^2\,dx .
\]
{
For sufficiently regular solutions with adequate velocity decay and integrability of the stochastic terms, the following balance laws hold; see \cite{CSZZ2025}. The mass is conserved pathwise,
} 
\[
\mathcal M[f(t)]=\mathcal M[f(0)], \qquad\text{a.s.}
\]
Moreover, It\^o's formula and integration by parts yield the momentum and energy balance laws
\begin{align}\label{mom-evo-law}
d\mathcal P[f(t)]
=\sum_{k=1}^{K}\int_{\mathbb T^{\text d}}\rho(t,x)\sigma_k(x)\,dx\,d\beta_k(t),
\end{align}
and
\begin{align}\label{ene-evo-law}
d\mathcal H[f(t)]
=\sum_{k=1}^{K}\int_{\mathbb T^{\text d}}J(t,x)\cdot\sigma_k(x)\,dx\,d\beta_k(t)
+\frac12\int_{\mathbb T^{\text d}}\operatorname{Tr}(\sigma\sigma^\top)(x)\rho(t,x)\,dx\,dt .
\end{align}
{
In particular, the total momentum is conserved in expectation, whereas the expected total energy evolves according to the It\^o correction term in \eqref{ene-evo-law}. These quantities describe the particle content, bulk momentum, and combined kinetic and electrostatic energy of the plasma. The momentum and energy laws quantify the effects of the random forcing: individual realizations exchange momentum with the external field, while the nonnegative energy drift represents energy input from the noise for nonnegative densities \cite{CSZZ2025}. Reproducing these laws is therefore important for distinguishing the physical response to stochastic forcing from artificial momentum drift or numerical heating.
}

In recent years, the stochastic Vlasov--Poisson equation has attracted increasing attention in the numerical analysis of stochastic partial differential equations, owing to its ability to model unresolved fluctuations and random forcing in kinetic systems. 
{For instance, \cite{BC} studied splitting integrators for linear Vlasov equations with several types of stochastic perturbations, including transport noise. \cite{CSZZ2025} developed a dynamic-domain semi-Lagrangian method and investigated the evolution of mass, momentum, and energy.}
At the same time, the numerical approximation of the deterministic Vlasov--Poisson equation remains challenging because of the high dimensionality of phase space; see \cite{S_lecturenotes,MR4715229,MR4329985} and the references therein. 
Among model reduction techniques, dynamical low-rank (DLR) approximation \cite{KL07} has emerged as one of the most effective approaches for kinetic equations. The Vlasov--Poisson equation has become a benchmark problem for DLR, leading to efficient reduced-order models capable of substantially lowering the computational complexity while maintaining good accuracy even for moderate ranks; see \cite{EKKMQ, EL18, EOP, GQ, NR, UZ}. More recently, considerable effort has been devoted to constructing conservative DLR integrators that preserve the physical structure and invariants of the continuous problem while remaining robust with respect to the rank approximation, a distinguishing feature of DLR methods; see \cite{EL18_cons, EJ2021, EOS, KNW}.

Dynamical low-rank approximation is also becoming an increasingly active research direction for stochastic problems; see \cite{KNZ, KNZ1}. In this context, it is worth recalling that dynamically orthogonal (DO) methods have made fundamental contributions to low-rank approximations of random partial differential equations and have laid the foundations for many subsequent developments in stochastic low-rank methodologies; see \cite{FL,MNZ,SL}.

The purpose of this paper is to construct dynamical low-rank approximations that respect the structural properties of \eqref{sto-Vlasov}. 
{We extend the conservative Petrov--Galerkin framework of \cite{EJ2021}, in which the low-rank evolution is determined by testing the equation residual against a prescribed family of functions; see \eqref{galerkin}.
We include time-independent velocity functions spanning $1$, $v_1,\ldots,v_\text{d}$, and $|v|^2$, so that testing against them recovers the velocity moments defining mass, momentum, and kinetic energy. This yields continuous Stratonovich low-rank dynamics with the corresponding balance laws.

We then construct two stochastic integrators within the augmented basis-update Galerkin (BUG) framework \cite{KLW,CL,CKL}. This framework updates the spatial and velocity bases before evolving the coefficient matrix in the resulting approximation spaces. Following \cite{EOS}, basis augmentation combines the old and newly computed basis functions, while conservative rank truncation reduces the enlarged representation without changing the selected velocity moments; see Section~\ref{aug_trunc}. For deterministic kinetic problems, augmented BUG integrators with forward Euler substeps and conservative truncation recover local conservation without additional spatial enrichment \cite{EKS}. We investigate the stochastic moment balances within this framework.

}

To construct the time integrators, we consider two strategies based on the two formulations of the original stochastic equation, we consider two strategies. The first is based on the Euler--Maruyama method applied to the equivalent It\^o formulation \eqref{sto-Vlasov-ito}, whereas the second uses the stochastic Heun method, motivated by its convergence to the Stratonovich formulation \eqref{sto-Vlasov}. 
These choices are not ad hoc alternatives; rather, they represent two popular ways to discretize the same stochastic model and allow us to see how the It\^o and Stratonovich viewpoints interact with the conservative low-rank approximation. Both schemes use the same explicit treatment of the deterministic transport and the same conservative rank-truncation procedure. This allows us to isolate the effect of the stochastic discretization on the discrete moment balances.

{
The discrete balance laws require control of both the stochastic correction terms and the moment errors introduced by projection and truncation.  First, the It\^o correction involves second-order velocity derivatives, whereas the Heun predictor produces quadratic products of Brownian increments. Although these terms have vanishing mass and first velocity moments, their second velocity moments contribute to the energy balance and must be retained through the low-rank updates. Second, inclusion of the velocity moment modes alone does not ensure that the spatial projection reproduces their evolution: the complete moment increments must lie in the augmented spatial space, and conservative truncation must leave the resulting moment densities unchanged. Moreover, the updated basis depends on the current Brownian increments, so mean-zero Brownian increments alone do not justify conservation in expectation.}

{
To establish the discrete balance laws, we compare two stages of the BUG update. The $K$-step evolves the spatial coefficient functions $K_j=\sum_iX_iS_{ij}$ associated with the velocity basis functions in \eqref{low-rank-formula}, whereas the $S$-step updates the coefficient matrix $S$ in the augmented bases; see \eqref{eq:unconv_x} and \eqref{eq:unconv_S}. We show that these two stages produce the same updates for the fixed velocity moments. The difference between the old and updated spatial coefficient functions represents the corresponding moment increment. Since both functions belong to the augmented spatial space, this increment is reproduced exactly by the spatial projection. Conservative rank truncation then leaves these moments unchanged by compressing only velocity components orthogonal to the fixed modes; see Section~\ref{aug_trunc}. Under compatible spatial discretization, the resulting balances yield pathwise mass conservation and conservation of total momentum in expectation. The energy analysis identifies two stochastic corrections with the same conditional expectation and separates their contributions from the electric-field residual \eqref{eq:electric-residual}.
}

The paper is organized as follows. Section 2 derives the continuous projected low-rank dynamics and explains how the fixed velocity modes reproduce the continuous mass, momentum, and energy laws. Section  3 introduces the Euler--Maruyama and Heun--Stratonovich BUG schemes, including basis augmentation and conservative rank truncation. Section 4 proves the corresponding discrete evolution laws. Section 5 discusses implementation details and presents numerical experiments, and Section 6 concludes with future directions.

\section{Dynamical Low-Rank Approximation}

{
We first derive the continuous factor equations from the Stratonovich formulation \eqref{sto-Vlasov}. We then describe the projection of full-equation It\^o increments used to construct the Euler--Maruyama integrator in Section~\ref{em}.
}

In this section, we derive a dynamical low-rank approximation for the stochastic Vlasov--Poisson equation \eqref{sto-Vlasov}. The construction is designed to reproduce the structural properties of the full model, including pathwise mass conservation, conservation of momentum in expectation, and the mean energy evolution law.
  
\subsection{Continuous Dynamics of Low-Rank Factors}

Following \cite{EJ2021, EOS}, we seek an approximate solution of the form
\begin{align}\label{low-rank-formula}
f(t,x,v)=f_{0v}(v)\sum_{i,j=1}^{r} X_i(t,x) S_{ij}(t) V_j(t,v),
\end{align}
where $i,j=1,\ldots,r$ and $r\ge m$ is the rank of the approximation with $m$ the number of fixed modes specified below. 
The spatial factors $X_i(t,\cdot)\in L^2(\mathbb T^{\text d})$ depend only on $x$, the velocity factors $V_j(t,\cdot)\in L^2(\mathbb R^{\text d};f_{0v})$ depend only on $v$, and the coefficient matrix $S(t)=(S_{ij}(t))\in\mathbb R^{r\times r}$ is time dependent. Here, $L^2(\mathbb R^{\text d};f_{0v})$ denotes the weighted $L^2$ space with positive weight $f_{0v}$.
The corresponding inner products are
$$ \langle X_i(t,\cdot), X_j(t,\cdot) \rangle_x = \int_{\mathbb T^{\text d}} X_i(t,x) X_j(t,x) \,dx$$
and
$$\langle V_i(t,\cdot), V_j(t,\cdot) \rangle_v = \int_{\mathbb R^{\text d}} f_{0v}(v) V_i(t,v) V_j(t,v) \,dv.$$

As in \cite{EJ2021, EOS}, 
the weight $f_{0v}$ is chosen so that the functions $1$, the components of $v$, and $|v|^2$ belong to $L^2(\mathbb R^{\text d};f_{0v})$. This holds, for example, for weights that decay sufficiently fast as $\|v\|_{\mathbb R^{\text d}}\to\infty$.
As in the deterministic result of \cite{EJ2021}, including these moment functions in the approximation space allows the dynamical low-rank approximation to reproduce the corresponding conservation and balance laws.

To accomplish this, we fix the first $m\ge \text{d}+2$ $v$-basis functions, which will remain unchanged under the low-rank dynamics. More precisely, for the energy evolution used below, we take \(m\ge \text d+2\); \(m\ge \text d+1\) is sufficient only if the energy mode is not included.
We denote the fixed $v$-basis functions as
\begin{align*}
 U_a(v)=V_a(t,v), \quad 1\le a\le m,
 \end{align*}
 and the remaining functions as
 \begin{align*}
 W_{p}(t,v)=V_p(t,v),\quad m<p\le r.
\end{align*}
\noindent Define the solution manifold $\mathcal M_r$ by 
\begin{align*}
\mathcal M_r
&=\Big\{f\in L^2(\mathbb T^{\text d}\times \mathbb R^{\text d};1/f_{0v}):
f(x,v)=f_{0v}(v)\sum_{i,j=1}^{r} X_i(x) S_{ij} V_j(v),\\
&\qquad S=(S_{ij})_{i,j\le r}\ \text{is invertible},\quad
X_i\in L^2(\mathbb T^{\text d}),\quad V_j\in L^2(\mathbb R^{\text d};f_{0v}),\\
&\qquad \langle X_i,X_j \rangle_x=\delta_{ij},\quad \langle V_i,V_j \rangle_v=\delta_{ij}, \quad V_a = U_a \quad 1 \leq a \leq m  \Big\},
\end{align*}
and its tangent space at $f$ by 
\begin{align*}
\mathcal T_f \mathcal M_r
&=\Big\{\dot f\in L^2(\mathbb T^{\text d}\times \mathbb R^{\text d};1/f_{0v}):
\dot f(x,v)=f_{0v}(v)\sum_{i,j=1}^r
\bigl(\dot X_iS_{ij}V_j+X_i\dot S_{ij}V_j\bigr)  \\ & + f_{0v}(v) \sum_{i= 1}^r \sum_{j = m+1}^r X_iS_{ij}\dot W_j,
\quad \dot S\in \mathbb R^{r\times r},\quad
\dot X_i\in L^2(\mathbb T^{\text d}),\quad
\dot W_j\in L^2(\mathbb R^{\text d}; f_{0v}),\\
&\qquad \langle X_i,\dot X_j\rangle_x=0,\quad \langle V_i,\dot W_j \rangle_v=0\Big\}.
\end{align*}
Here $\delta_{ij}$ is the Kronecker delta.

The stochastic evolution is required to remain on $\mathcal M_r$. 
The orthogonality and gauge conditions $\langle X_i, \star dX_j\rangle_x=0$
and $\langle V_i,\star dV_k\rangle_v=0$ for $i,j,k\le r$ are imposed. We impose the following Petrov--Galerkin condition on the residual of the stochastic Vlasov--Poisson equation to derive the equations of motion for the low-rank factors:
\begin{align}\label{galerkin}
(\nu/f_{0v},df-\mathrm{RHS})_{xv}=0,   
\end{align}
for every $\nu\in \mathcal{T}_f\mathcal M_r$, where
\[
\mathrm{RHS}
=-\bigl(v\cdot\nabla_x f+E(t,x)\cdot\nabla_v f\bigr)dt
-\sum_{k=1}^{K}\sigma_k(x)\cdot\nabla_v f\star d\beta_k(t).
\]
The inner products $(\cdot,\cdot)_{xv}$, $(\cdot,\cdot)_x$, and $(\cdot,\cdot)_v$ are standard $L^2$ inner products, while $\langle\cdot,\cdot\rangle_v$ and $\langle\cdot,\cdot\rangle_{xv}$ denote the corresponding $f_{0v}$-weighted inner products.

Suitable choices of $\nu\in \mathcal{T}_f\mathcal M_r$ isolate the equations for the individual low-rank factors. The test functions used in the deterministic setting remain valid in the stochastic case; we therefore proceed analogously and refer to \cite{EJ2021} for additional details.

By the chain rule for Stratonovich integrals,
\begin{align}\label{chain-rule}
 d f=f_{0v}(v)\Big[\sum_{ij} \star d (X_i S_{ij}) V_j- \sum_{ij} X_i \star d S_{ij} V_j+ \sum_{ij} X_i \star d (S_{ij} V_j)\Big].
\end{align}

\subsubsection{Stratonovich  factor equations}

Set $K_k=\sum_iX_iS_{ik}$. The equation below evolves these spatial coefficient functions; the factors $X_i$ are recovered subject to the orthonormality and gauge conditions.
 
\textbf{Equations for $X_i$.}
Taking the $L^2$ inner product $(\cdot,\cdot)_v$ in \eqref{chain-rule} corresponds to the Petrov--Galerkin condition with the test function $\nu_k= f_{0v} \chi(x) V_k$, $k\le r$, where $\chi$ is an arbitrary integrable function of $x$. This gives
\begin{align*}
(V_k\chi(x),f_{0v}\sum_{ij}\star dX_i S_{ij}V_j+X_i\star dS_{ij}V_j+X_iS_{ij}\star d V_j)_{xv}=(V_k\chi(x),\mathrm{RHS})_{xv}.
\end{align*}

Using orthogonality and the gauge condition, we obtain
\begin{align}\nonumber
\sum_{i}d(X_iS_{ik})
&=(V_k,-v\cdot \nabla_x f-E\cdot \nabla_v f)_{v} dt
-(V_k,\sum_{l=1}^K\sigma_l(x)\cdot \nabla_v f)_v \star d \beta_l(t)\\\nonumber
&=-\sum_{ij}\langle V_k, v V_j \rangle_v\nabla_x X_iS_{ij} dt-\sum_{ij}E\cdot (V_k,\nabla_v(f_{0v}V_j))_v X_iS_{ij} dt\\\nonumber
&\quad -\sum_{ij} \sum_{l=1}^K \sigma_l(x)\cdot  (V_k, \nabla_v (f_{0v}V_j))_v X_iS_{ij}\star d \beta_l(t)\\\label{low-eq-1}
&=-\sum_{ij}c^1_{kj}{\color{red}\cdot}\nabla_x X_iS_{ij}dt -\sum_{ij} (c^2_{kj} \cdot E)X_iS_{ij}dt\\\nonumber
&\quad -\sum_{ij} \sum_{l=1}^K \sigma_l(x) {\color{red}\cdot}   c^2_{kj}X_iS_{ij}\star d \beta_l(t).
\end{align}
where $c^1_{kj}:= \langle V_k, v V_j \rangle_v$, $c^2_{kj}:= ( V_k, \nabla_v (f_{0v}V_j ))_v$.

\textbf{Equations for $W_p$.}
Taking the inner product $\langle \cdot,\cdot\rangle_x$ in \eqref{chain-rule} with $\sum_{i}X_iS_{iq}$ and using orthogonality and the gauge
conditions, we obtain
\begin{align}\nonumber
&\sum_{ip} S_{iq} S_{ip}\star d{W}_p+\sum_{ij}S_{iq}\star d{S}_{ij}V_j\ = \frac 1{f_{0v}}\sum_{i}S_{iq}\langle X_i,-v\cdot \nabla_x f-E\cdot \nabla_v f\rangle_{x} dt \\\nonumber
&+\frac 1{f_{0v}}\sum_{i}S_{iq} \langle X_i,-\sum_{l=1}^K\sigma_l(x)\cdot \nabla_v f\star d\beta_l(t)\rangle_x \\\label{low-eq-2}
&=-\frac 1{f_{0v}} \sum_{i} S_{iq} \sum_{kl} d^1_{ik}[E] \cdot \nabla_v (f_{0v} S_{kl} V_l) dt-\sum_{i} S_{iq}\sum_{kl}(d^2_{ik}\cdot v ) S_{kl}V_l dt
\\\nonumber
&-\frac 1{f_{0v}}\sum_i S_{iq} \sum_{n=1}^K \sum_{kl}   (d^3_{ink}[\sigma_n] ) {\color{red}\cdot}   \nabla_v (f_{0v}S_{kl}V_l) \star d\beta_n(t),
\end{align} 
Here $q=m+1,\ldots,r$ because only the moving velocity modes $W_q$ evolve in this equation; the fixed modes $U_a$, $a=1,\ldots,m$, remain unchanged. The index $p$ also ranges over $m+1,\ldots,r$, while $i,k,l$ range over $1,\ldots,r$.
Here $d^1_{ik}[E] = \langle X_i, E X_k\rangle_x$, $d^2_{ik} = \langle X_i, \nabla_x X_k \rangle_x$ and $ d^3_{ink}[\sigma_n] = \langle X_i, \sigma_n X_k\rangle_x$.
The above is equivalent to imposing the Petrov--Galerkin condition with respect to the test functions $\nu_q = f_{0v}\xi(v) \sum_i X_i S_{iq}$, where $\xi$ is an arbitrary function of $v$.

\textbf{Equations for $S_{ij}$.}
Taking the space--velocity inner product $(\cdot,\cdot)_{xv}$ in \eqref{chain-rule} with $X_kV_l$ and using the gauge
condition, we similarly obtain that for the Stratonovich equation, 
\begin{align}\nonumber
  d S_{kl}
& =-\sum_{ij}\langle X_k,\nabla_x X_i\rangle_x\cdot \langle V_l, v V_j\rangle_v S_{ij} dt
-\sum_{ij}\langle X_k,E X_i\rangle_x\cdot (V_l, \nabla_v (f_{0v}V_j))_v S_{ij}dt\\\label{low-eq-3}
&-\sum_{n=1}^K\sum_{ij}\langle X_k,\sigma_n  X_i\rangle_x\cdot (V_l, \nabla_v (f_{0v}V_j))_v S_{ij}\star d\beta_{n}(t).
\end{align}

\subsubsection{Projected It\^o  factor equation}
{
Applying above same procedures  to \eqref{sto-Vlasov-ito}, the It\^o formulation of the full equation,  we obtain
\begin{align*}
\sum_{i}d(X_iS_{ik})
&=-\sum_{ij}c^1_{kj}\nabla_x X_iS_{ij}dt -\sum_{ij} (c^2_{kj} \cdot E)X_iS_{ij}dt\\
&\quad -\sum_{ij} \sum_{l=1}^K \sigma_l(x)  c^2_{kj}X_iS_{ij} d \beta_l(t)+\dfrac{1}{2}\sum_{ln}(\sigma\sigma^{\top})_{ln}(x)  \sum_{ij} X_i S_{ij} c^3_{lnjk}  dt, 
\end{align*}
where  $c^3_{lnjk} = \int \partial_{v_lv_n}(f_{0v}V_j)V_k 
     dv$.

Using $[\log(g)]'=g'/g$, { we also have that 
\begin{align*}
&\sum_{ip} S_{iq} S_{ip}\star d{W}_p+\sum_{ia}S_{iq} \star d{S}_{ia}V_a\\&=\sum_i S_{iq}\Big[-\sum_{kl}(d^2_{ik}\cdot v ) S_{kl}V_l dt
-\sum_{kl} d^1_{ik}[E]\cdot [ \nabla_v(S_{kl} V_l) + \nabla_v(\log f_{0v}) S_{kl} V_l] dt \\
&-\sum_{n=1}^K\sum_{kl} d^3_{ink}[\sigma_n(x)]
[ \nabla_v(S_{kl} V_l) + \nabla_v(\log f_{0v}) S_{kl} V_l]d\beta_n(t)\\
&+\dfrac{1}{2 f_{0v}} \sum_{a,n=1}^{\mathrm{d}} \sum_{kl}
d^4_{anik}[\sigma\sigma^{\top}]
\partial_{v_av_n}(f_{0v}S_{kl}V_l)dt\Big],
\end{align*}
where $d^4_{anik}[\sigma \sigma^{\top}] = \langle (\sigma \sigma^{\top})_{an}(x)  X_i,  X_k \rangle_x$.
Similarly, for the equation of $S$, we 
have 
\begin{align}\label{low-eq-4}
  d S_{kl}&=-\sum_{ij} d^2_{ki}\cdot  c^1_{lj} \;S_{ij} dt
-\sum_{ij}d^1_{ki}[E]\cdot  c^2_{lj}\;S_{ij}dt-\sum_{n=1}^K\sum_{ij}d^3_{ink}[\sigma_n(x)]\cdot c^2_{lj} S_{ij}d\beta_{n}(t)\\\nonumber
&+\dfrac{dt}{2}  \sum_{a,n=1}^{K} \sum_{ij} S_{ij} \;d^4_{anik}[\sigma\sigma^{\top}] \cdot c^3_{anjl}.
\end{align}
where $c^3_{anjl}= (\partial_{v_av_n}(f_{0v}V_j),V_l)_v$. 

Note that the projected It\^o spatial-factor equation is used for the Euler-- Maruyama-based BUG integrator later. Stratonovich--It\^o conversion and low-rank projection do not generally commute. Let $\mathcal P_f$ denote the constrained tangent-space projector and set $B_sf=-\sigma_s\cdot\nabla_vf$. On a smooth fixed-rank manifold, the projected noise field $G_s(f)=\mathcal P_fB_sf$ has It\^o correction
\[
\frac12\sum_s DG_s(f)[G_s(f)]
=\frac12\sum_s\left(
D\mathcal P_f[G_s(f)]B_sf+
\mathcal P_fB_s\mathcal P_fB_sf\right),
\]
which need not equal $\frac12\sum_s\mathcal P_fB_s^2f$.
Thus, the projected full-equation increments above are not identified with an equivalent It\^o form of \eqref{low-eq-1}--\eqref{low-eq-3}. The continuous balance laws below concern the Stratonovich dynamics; the discrete balance laws are established separately for the two BUG schemes.

}

}
\begin{remark}
The main advantage of the low-rank approximation is that, instead of treating the full problem \eqref{sto-Vlasov} in dimension $2\mathrm{d}$  (with $\mathrm d \leq 3$), one solves {$r$ spatial equations and $r-m$ moving velocity equations, each in dimension $\text{d}$,} x`together with {$r^2$ scalar stochastic equations for the coefficient matrix}. When $r$ is considerably smaller than {the number of spatial and velocity degrees of freedom}, as is often the case in applications, the dynamical low-rank approximation can substantially reduce the memory requirements and computational cost.
\end{remark}

\subsection{{Structure-preserving properties of the low rank dynamics}}
In this subsection, we present the continuous conservation properties of the dynamical low-rank
approximation derived above. A typical choice of the positive weight function $f_{0v}$ is a Gaussian-type function, such as  $e^{-\frac {{|v|^2}} 2}$,  {or variants thereof}; see \cite{EJ2021, EOS}.  We use \(\|\cdot\|_v\) for the weighted velocity norm induced by \(\langle\cdot,\cdot\rangle_v\). Set
\[
{
\alpha_2=\frac{\langle |v|^2,1\rangle_v}{\langle 1,1\rangle_v},
\qquad
\chi_2(v)=|v|^2-\alpha_2 .
}
\]
The adopted normalized fixed modes are \(U_1=1/\|1\|_v\), \(U_{i+1}=v_i/\|v_i\|_v\) for \(1\le i\le\text d\), and \(U_{\text d+2}=\chi_2/\|\chi_2\|_v\). For a standard Gaussian in \(\text d\) dimensions, \(\alpha_2=\text d\).
{We assume that $f_{0v}$ is even in each velocity component, as is the case for the Gaussian, so that $\langle 1,v_i\rangle_v=0$, $\langle v_i,v_j\rangle_v=0$ for $i\ne j$, and $\langle \chi_2,v_i\rangle_v=0$. Together with $\langle\chi_2,1\rangle_v=0$, which holds by the choice of $\alpha_2$, this makes $U_1,\dots,U_{\text d+2}$ orthonormal in $L^2(\mathbb R^{\text d};f_{0v})$.}
This choice will ensure that \eqref{low-eq-1}--\eqref{low-eq-3} preserve the structural properties of the original problem.

The mass density function is given by
\begin{align*}
\rho=\sum_{ij}X_iS_{ij} \langle V_j,1\rangle_v={\|1\|_v}\sum_{i} X_iS_{i1}.
\end{align*}
This identity, together with \eqref{low-eq-1}, implies
\begin{align*}
d \rho&= {\|1\|_v}\sum_{i} d (X_iS_{i1})=- \left(\nabla_x \cdot \int_{\mathbb{R}^{{\text d}}} v f \,dv + E\cdot \int_{\mathbb{R}^{{\text d}}} \nabla_v f \,dv \right) dt\\ \nonumber 
&-\sum_{l=1}^K \sigma_l \cdot  \int_{\mathbb{R}^{{\text d}}} \nabla_v f \,dv \star d\beta_l(t).  
\end{align*}
Thus,
\begin{align}\label{local-mass}
d \rho(t,x)+\nabla_x \cdot J(t,x)dt=0,
\end{align}
which, upon integration over the spatial domain, implies the mass conservation law
\[
\mathcal M[f(t)]=\mathcal M[f(0)], \qquad \text{a.s.}
\]

Since $U_{j+1}=v_j/\|v_j\|_v$ for $1\le j\le d$, the momentum components are $J_j=\|v_j\|_v\sum_iX_iS_{i,j+1}$. 
From \eqref{low-eq-1}, it holds that 
\begin{align}\label{local-moment}
  d J  = \left({\|v_j\|_v}d \sum_{i}X_iS_{i(j+1)} \right)_{j=1}^{{\text d}} 
  &=-\nabla_x\cdot \left(\int_{\mathbb R^{\text d}}v\otimes v f \,dv \right) dt{+}E \rho \,dt \\ \nonumber & +\sum_{l=1}^K\sigma_l(x) \rho  \,  \star d\beta_l(t).
\end{align}
By the Poisson equation,
$$\int_{\mathbb{T}^d} E \rho \,dx=\int_{\mathbb{T}^d} E (1+ \text{div}(E)) dx=-\int_{\mathbb{T}^d} \nabla u \,dx+\int_{\mathbb{T}^d}  \text{div}(\nabla u) \nabla u \, dx=0.$$
Note that the stochastic term in \eqref{local-moment}  has the same It\^o and Stratonovich representations. Assuming square integrability of its coefficient, spatial integration and expectation
 yield  conservation of momentum in expectation, i.e.,
\begin{align*}
\mathbb E[\mathcal P[f(t)]]=\mathbb E[\mathcal P[f(0)]].
\end{align*}

Finally, for orthogonality reasons, the quadratic mode cannot be chosen simply proportional to $|v|^2$. We take \(U_{\text d+2}=\chi_2/\|\chi_2\|_v\), so that $\langle U_{{\text d}+2},U_i\rangle_v=0$ for $1\le i\le {\text d}+1$ by construction.
We use the representation
\begin{align*}
|v|^2=\yedit{\|\chi_2\|_vU_{{\text d}+2}+\alpha_2\|1\|_vU_1}.
\end{align*}  
A direct computation gives
\begin{align*}
\int_{\mathbb{R}^{{\text d}}} |v|^2 f \,dv
&=\sum_{ij} X_iS_{ij}\Bigl({\|\chi_2\|_v}\langle U_{{\text d}+2},V_j\rangle_v+ {\alpha_2\|1\|_v}\langle U_1, V_j\rangle_v\Bigr)\\
&={\|\chi_2\|_v}\sum_{i} X_iS_{i({\text d}+2)}+{\alpha_2\|1\|_v}\sum_{i} X_iS_{i1}.
\end{align*}
Using the definition of the energy $\mathcal H$ and the chain rule, we further obtain (see also \cite[Proposition 1]{CSZZ2025})
\begin{align}\nonumber
&d \mathcal H(f(t))\\\nonumber
&=\frac 12 {\|\chi_2\|_v}
\Bigl(U_{{\text d}+2},
-(v\cdot \nabla_x f+E\cdot \nabla_v f)dt
-\sum_{k=1}^{K} \sigma_k(x) \cdot \nabla_v f \star d\beta_k(t)
\Bigr)_{xv}\\\nonumber
&+\frac 12 {\alpha_2\|1\|_v}
^2\Bigl(U_1,
-(v\cdot \nabla_x f+E\cdot \nabla_v f)dt
-\sum_{k=1}^{K} \sigma_k(x) \cdot \nabla_v f \star d\beta_k(t)
\Bigr)_{xv}
+(E, d E)_{xv}\\\label{local-energy}
&=\sum_{l=1}^K (\sigma_l,J(t))_x  \,d\beta_l(t)
+\frac 12 (\operatorname{Tr}(\sigma \sigma^\top),\rho(t))_{x} dt.
\end{align}
These calculations show that the continuous low-rank approximation preserves the mass, momentum, and energy balance laws.

\section{Low-Rank Stochastic BUG Integrators}

In this section, we present dynamical low-rank integrators for the stochastic Vlasov--Poisson equation. We first discretize in space using centered finite differences with boundary conditions chosen to preserve the main structural properties of the continuous problem, in particular the vanishing of discrete divergence sums and the discrete analogue of the integration-by-parts identities. Further details on the discretization and implementation are provided in Section \ref{implementation}.

We now describe the fully discrete low-rank schemes. The method is based on an augmented
unconventional integrator, also known as a BUG integrator. The algorithm updates the spatial
basis, the moving velocity basis, and the coefficient matrix through three projected steps, followed
by conservative truncation.

\subsection{Euler--Maruyama-Based BUG Integrator}
\label{em}

A natural first approach to constructing a dynamical low-rank integrator for the stochastic Vlasov--Poisson equation is to proceed formally as in the unconventional integrator for the deterministic equation \cite{EJ2021}.  This method is designed for the equivalent It\^o formulation of the original equation \eqref{sto-Vlasov-ito}.

Define 
\begin{equation}
\label{dfn_EM}
    \mathcal{D}(f^n) = -v \cdot \nabla_x f^n - E^n \cdot \nabla_v f^n
\end{equation} 
and
\begin{equation}
\label{sfn_EM}
   \mathcal{S}_{\mathrm{EM}}(f^n) = - \sum_{j = 1}^K \sigma_j(x) \cdot \nabla_v f^n \Delta\beta_n^j+ \frac{1}{2} \tau \; \nabla_v \cdot (\sigma \sigma^{\top}(x) \cdot \nabla_v f^n).
\end{equation}
where $\tau>0$ is the time step, $t_n = n \tau$, $N = T/\tau \in \mathbb N,$ and $\Delta\beta_n^j=\beta_j(t_{n+1})-\beta_j(t_n)$, $n = 0,\dots,N-1$, denotes the increment of the $j$-th Wiener process over the time interval $[t_n,t_{n+1}]$.
We project the Euler--Maruyama increment of the full It\^o equation onto the current factor spaces to define the following basis update: 
\begin{align}
\label{eq:unconv_x}
K_k^{n+1} &= K_k^n+\tau \left(V_k^n, \mathcal{D}(f^n) \right)_v + \left(V_k^n,    \mathcal{S}_{\mathrm{EM}}(f^n) \right)_v, \qquad 1 \leq k \leq r,\\ \nonumber
\end{align}
where
$$K_k^n=\sum_i X_i^n S_{ik}^n,$$
\begin{align}
\label{eq:unconv_w}
L_q^{n+1} & = L_q^n+\frac{\tau}{f_{0v}} \sum_i S^n_{iq} {\langle} X_i^n, \mathcal{D}(f^n){\rangle} _x-\tau \sum_{il}S_{iq}^n \left(X_i^n V_l^n , \mathcal{D}(f^n)\right)_{{xv}} V_l^n,\\ 
& +\frac{1}{f_{0v}} \sum_i S^n_{iq} {\langle} X_i^n,   \mathcal{S}_{\mathrm{EM}} (f^n){\rangle} _x- \sum_{il}S_{iq}^n \left(X_i^n V_l^n ,   \mathcal{S}_{\mathrm{EM}} (f^n)\right)_{{xv}} V_l^n, \quad m < q \leq r. \nonumber
\end{align}
with
$$L^n_q=\sum_{ip} S^n_{iq}S^n_{ip}W_p^{n}, \qquad \text{and} \quad L^{n+1}_q=\sum_{ip} S^n_{iq}S^n_{ip}W_p^{n+1}.$$ 
After computing $K^{n+1}$ and $L^{n+1}$, we orthogonalize them together with the bases at $t_n$, as described in Section \ref{aug_trunc}.
The quantities
$\widetilde X^{n+1},\widetilde V^{n+1}$
represent intermediate augmented bases. The final bases
$X^{n+1},W^{n+1}$
are obtained after the conservative rank truncation described below.

The augmented approximation is 
$$f^{n+1}_{aug} = f_{0v}\textstyle\sum_{kijl}\widetilde X^{n+1}_k M_{ki}S^n_{ij} N^\top_{jl} \widetilde V^{n+1}_l,$$
where 
\begin{equation}
\label{def_MN}
 M_{ki}=\langle  \widetilde X_k^{n+1}, X_i^n\rangle_x, \quad N^\top_{jl}=\langle  V_j^{n}, \widetilde V_l^{n+1}\rangle_v.
 \end{equation}

The coefficient matrix is then updated via \begin{equation}
\widetilde S_{kl}^{n+1}=\sum_{ij} M_{ki} S^n_{ij} N^{\top}_{jl}+ \tau \left(\widetilde X_k^{n+1} \widetilde V_l^{n+1}, \mathcal{D}(f^{n+1}_{aug}) \right)_{xv} +\left(\widetilde X_k^{n+1} \widetilde V_l^{n+1},    \mathcal{S}_{\mathrm{EM}}(f^{n+1}_{aug}) \right)_{xv}.  \label{eq:unconv_S}
\end{equation}
The term $MS^nN^{\top}$ represents the projection of the solution at time $t_n$ onto the updated basis functions at time $t_{n+1}$, namely, 
\begin{equation}
    \sum_{ij} X_i^n S_{ij}^n V_j^n = \sum_{kijl} \widetilde X_k^{n+1}M_{ki} S_{ij}^n N^\top_{jl} \widetilde V_l^{n+1}.
    \label{eq:MSN}
\end{equation}
{The equality holds thanks to the augmentation of Section \ref{aug_trunc}.}
A key advantage of the unconventional integrator is its robustness with respect to small singular values in the rank approximation; see \cite{CL, CKL, EJ2021}. This robustness follows from the fact that no inverse of $S$ is required.

\subsection{Heun--Stratonovich BUG Integrator}
\label{heun}
In this subsection, we present a variant of the BUG integrator in which the noise increment uses a Heun predictor--corrector and the deterministic drift uses forward Euler. 
The motivation for this choice is that the Heun discretization is consistent with the Stratonovich interpretation in the limit of vanishing time step; see \cite{Lordbook2014}. It is also explicit, which keeps the computational cost moderate.
This integrator is directly designed for the Stratonovich formulation.  

The augmented unconventional integrator (BUG) reads
\begin{equation}
\label{k_heun}
K_k^{n+1} = K_k^n
+ \tau \left( V_k^n , \mathcal{D}(f^n) \right)_v
+ \left( V_k^n ,    \mathcal{S}_{\mathrm{Heun}}(f^n) \right)_v,
\qquad 1 \le k \le r, 
\end{equation}
\begin{align}
\label{l_heun}
L_q^{n+1} &= L_q^n
+ \frac{\tau}{f_{0v}} \sum_i S_{iq}^n \langle X_i^n , \mathcal{D}(f^n) \rangle_x
- \tau \sum_{i l} S_{iq}^n \left( X_i^n V_l^n , \mathcal{D}(f^n) \right)_{xv} V_l^n \nonumber \\
& + \frac{1}{f_{0v}} \sum_i S_{iq}^n \langle X_i^n ,  \mathcal{S}_{\mathrm{Heun}}(f^n) \rangle_x
- \sum_{i l} S_{iq}^n \left( X_i^n V_l^n ,  \mathcal{S}_{\mathrm{Heun}}(f^n) \right)_{xv} V_l^n,
\quad m < q \le r. 
\end{align}
The low-rank factors are given by
\[
K_k^n = \sum_i X_i^n S_{ik}^n,
\qquad
L_q^n = \sum_{i p} S_{iq}^n S_{ip}^n W_p^n,
\qquad
L^{n+1}_q=\sum_{ip} S^n_{iq}S^n_{ip}W_p^{n+1}.
\]
The deterministic part $\mathcal{D}(f^n)$ is the same as in \eqref{dfn_EM}.
The stochastic  term is discretized using the Heun scheme, which is consistent with the Stratonovich formulation,
\begin{equation}
\mathcal{S}_{\mathrm{Heun}} (f^n)
=
-\frac{1}{2} \sum_{s=1}^K
\Big(
\sigma_s(x) \cdot \nabla_v f^n
+
\sigma_s (x) \cdot \nabla_v \widetilde f^{\,n+1}
\Big)\,\Delta \beta^s_n ,
\end{equation}
with predictor
\begin{equation}
\label{heun_pred} 
\widetilde f^{\,n+1}
=
f^n
-
\sum_{s=1}^K\sigma_s(x) \cdot \nabla_v f^n \, \Delta \beta^s_n .
\end{equation}
The coefficient-matrix update uses \eqref{eq:unconv_S} with $\mathcal S_{\mathrm{EM}}$ replaced by $\mathcal S_{\mathrm{Heun}}$ throughout.

The method described above is therefore an unconventional integrator with a Heun-based stochastic discretization.

\subsection{Basis augmentation and conservative truncation} 
\label{aug_trunc}
The purpose of basis augmentation is to ensure that the functions needed in the conservation proofs
are represented exactly in the updated approximation spaces.  It is known that the standard unconventional integrator is not structure preserving \cite{EJ2021}. To preserve the evolution of the physical quantities, we augment the approximation spaces.
We define the augmented bases 
\[ \widetilde{X}^{n+1} = (\widetilde{X}_1^{n+1}, \widetilde{X}_2^{n+1}, \dots),
    \qquad \widetilde{V}^{n+1} = (\widetilde{V}_1^{n+1}, \widetilde{V}_2^{n+1}, \dots), \]
such that $(\widetilde{X}^{n+1}_j)_j$ is an orthonormal basis of {$\text{span}\{ X_i^n, K_i^{n+1}\}$} and $(\widetilde{V}^{n+1}_j)_j$ is an orthonormal basis of $\text{span}\{ U, \widetilde{W}^{n+1}\}$, where $\widetilde{W}^{n+1}  = [W^n_p,\, L_p^{n+1}] $.
Consequently, the projections satisfy
\begin{equation}
\label{proj_x}
P_{x,n+1}(X_j^n) \coloneqq \sum_i \langle  \widetilde{X}_i^{n+1}, X_j^n\rangle_x  \widetilde{X}_i^{n+1} = \sum_i M_{ij} \widetilde{X}_i^{n+1}  =   X_j^n,
\end{equation}
\begin{equation}
\label{proj_v}
P_{v,n+1}(V_l^{n}) \coloneqq \sum_k  \langle V_l^n, \widetilde{V}_k^{n+1}\rangle_v \widetilde{V}_k^{n+1} =
\sum_k N^{\top}_{lk} \widetilde{V}_k^{n+1} =  V_l^{n}.
\end{equation}

These projection identities {, together with the fact that $K_i^{n+1}$ belongs to the augmented spatial space,} are the algebraic mechanism behind the discrete mass, momentum, and
energy balances. {As in \cite{EKS}, we do not include the gradients $\nabla_x X_i^n$ in the augmented spatial basis.} 

After the augmented update, the rank is reduced by a conservative truncation introduced in \cite{EOS}. The fixed
velocity modes are kept exactly, and the remaining modes are compressed by a singular-value
decomposition. In this way, the method returns to rank at most $r$ while preserving the moment information
represented by the fixed modes. Indeed,  
$ \widetilde X^{n+1}$
is a matrix in $\mathbb R^{{n_x}\times {2r}}$, 
while 
$\widetilde V^{n+1}$ belongs to $\mathbb{R}^{{n_v}\times (2r-m)}$, {where $n_x$ and $n_v$ are the numbers of spatial and velocity grid points, respectively.}.
The conservative truncation proceeds as follows.

First we compute  $\widetilde{K} = \widetilde{X}^{n+1} \widetilde{S}^{n+1}$. Then we split it as $\widetilde{K}=[\widetilde{K}^{cons}, \widetilde{K}^{rem}],$
where $\widetilde{K}^{cons}$ consists of the first $m$ columns of $\widetilde{K}$ and $\widetilde{K}^{rem}$ consists of the remaining columns.
The QR decomposition gives 
$$ \widetilde{K}^{rem} = \widetilde{X}^{rem} \widetilde{S}^{rem},\; \widetilde{K}^{cons} = \widetilde{X}^{cons} \widetilde{S}^{cons}.$$
We then compute the truncated singular value decomposition 
$$ \widetilde{S}^{rem} \approx \hat{U}\hat{S} \hat{W}^T, $$
and retain the largest $r-m$ singular values.
Define 
$$X^{rem} = \widetilde{X}^{rem} \hat{U}, \quad S^{rem}= \hat{S}, \quad W^{n+1} =  \widetilde{W}^{n+1} \hat{W}.$$
The updated velocity basis is
$V^{n+1} = \left[U\; W^{n+1} \right],$
where the conserved modes $U$ remain fixed. 
Finally, define 
$ \hat{X} = \left[ \widetilde{X}^{cons} \; X^{rem} \right] $ and compute the QR decomposition
$$\hat{X} = X^{n+1} R.$$ 
The final coefficient matrix is 
$$S^{n+1} = R \begin{bmatrix} \widetilde{S}^{cons} & 0\\
0 & S^{rem}
\end{bmatrix}.$$
The approximation at time $t_{n+1}$ is therefore $$f^{n+1} = {f_{0v}}\sum_{ij} X^{n+1}_i S^{n+1}_{ij} V^{n+1}_j.$$

The truncation is conservative because the SVD compression is not applied to the
fixed modes associated with the conserved quantities. These modes are kept
unchanged, while only the remaining columns are compressed. Therefore, the
rank reduction does not modify the components required for the discrete
conservation properties.
\subsection{Rank adaptation}
Following the standard approach in dynamical low-rank approximation, we can adapt the rank according to the approximation error with respect to the solution; see \cite{CKL}.  The updated coefficient matrix $\widehat S$ is obtained after a singular value decomposition of $\widetilde{S}^{rem}$. We choose the new rank $\hat{r}$ such that

\[
\sum_{k=r+1}^{\widehat r}\widehat\sigma_k^2
\le
\theta^2,
\]
where $\{\widehat\sigma_k\}_{k=1}^{\widehat r}$ denote the singular values of $\widehat S$ and 
 $\theta>0$ is a prescribed tolerance that controls the error introduced by the low-rank approximation of the solution.

We remark that alternative rank-adaptation strategies have also been proposed in the literature, where the truncation criterion is based on quantities of physical interest rather than on the approximation error of the solution itself; see, e.g., \cite{EOS, SEKM}. In the present work, however, we employ the classical solution-based truncation criterion.

\section{Discrete Evolution Laws}
\label{sec:laws}
This section states the discrete evolution laws satisfied by the augmented BUG schemes. 
The proofs are written in general dimension $\mathrm{d}$.
{Throughout this section, the velocity integration-by-parts identities for the fixed polynomial moments and periodic spatial cancellation are assumed. In a fully discrete reading, integrals mean the chosen quadrature and the same identities must hold for the implemented operators. Statements about global momentum additionally assume $\int E^n\rho^n\,dx=0$. All discrete states at $t_n$ are $\mathcal F_{t_n}$-measurable, and stochastic terms are assumed integrable.}

We denote by $\rho^n$, $J^n$, and $e^n$, respectively, the discrete mass, momentum, and energy densities at time $t^n$.
For notational simplicity, from now on we omit the tilde from $X^{n+1}$, $S^{n+1}$, and $V^{n+1}$, using the fact that the truncation is conservative.
{We write $\mathcal I_\theta(f^n)=\tau\mathcal D(f^n)+\mathcal S_\theta(f^n)$, $\theta\in\{\mathrm{EM},\mathrm{Heun}\}$, for the total increment.}

 Written in coefficient
form for the Euler--Maruyama-based method, the $K$-step \eqref{eq:unconv_x} gives
\begin{align}
K_k^{n+1}
&=
K_k^n
-\tau\sum_j c_{kj}^1\cdot\nabla_xK_j^n
-\tau\sum_j c_{kj}^2\cdot E^nK_j^n
+\frac{\tau}{2}\sum_{j,l,n'}c_{ln'jk}^3\,A_{ln'}(x)K_j^n \nonumber\\
&\quad
-\sum_j c_{kj}^2\cdot
\left(\sum_{s=1}^K\sigma_s(x)\Delta\beta_n^s\right)K_j^n ,
\label{eq:K-coeff}
\end{align}
where $A(x)=\sigma(x)\sigma(x)^\top$ and $c^3_{ln'jk}=(V_k,\partial_{v_lv_{n'}}(f_{0v}V_j))_v$.
For the Heun-based method, expanding the predictor
\eqref{heun_pred}, the It\^o correction term in \eqref{eq:K-coeff} is replaced by the quadratic Heun term
\begin{equation}\label{eq:K-heun-term}
\frac12\sum_{s,h=1}^K\Delta\beta^s_n\Delta\beta^h_n\sum_{j,\ell}\big(c^2_{kj}\cdot\sigma_s\big)\big(c^2_{j\ell}\cdot\sigma_h\big)K^n_\ell .
\end{equation}
{Here the predictor is represented in the current basis, as in our implementation; with the exact predictor, the products $c^2_{kj}c^2_{j\ell}$ are replaced by the corresponding entries of $c^3$, and the two coincide on the fixed rows $k\le {\text{d}}+2$ used below, therefore all the results hold in both cases.}

\subsection{Local Mass Conservation Law}

\begin{theorem}
\label{thm_1}
Assume that the fixed velocity space contains the
constant mode ($m \geq 1$) and that the projection identities \eqref{proj_x}-\eqref{proj_v} hold.
Then the Euler--Maruyama BUG
integrator of Section~\ref{em} and the stochastic Heun BUG integrator of Section~\ref{heun} satisfy
the following local mass conservation law
\begin{equation}
    \label{cont_d_rho}
     \frac{\rho^{n+1}-\rho^n}{\tau} + \nabla_x \cdot J^n = 0,\; a.s.
     \end{equation}
   where $n=0,\ldots,N-1.$
In particular, the total mass is conserved, i.e., $\mathcal M[f^n]=\mathcal M[f^0]$ for all $0\le n\le N$.
\end{theorem}
\begin{proof}
We give a unified proof for the Euler--Maruyama and Heun BUG schemes.
Thanks to the augmentation of the basis, by exploiting  equations \eqref{proj_x} and \eqref{proj_v}, we obtain
\begin{equation} \label{eq:exact}
    \sum_{ijlb} X_i^{n+1} M_{ij} S_{jl}^n N^T_{lb} V_b^{n+1} =  \sum_{jl} X_j^{n}  S_{jl}^n  V_l^{n},
\end{equation}
that is, $f^{n+1}_{aug}=f^n$. {Let $U_a$, $a\le m$, be a fixed velocity mode. Since $\langle V_j^n,U_a\rangle_v=\delta_{ja}$,
the $S$-step \eqref{eq:unconv_S}, \eqref{proj_x} and \eqref{eq:exact} give
\begin{align}
K_a^{n+1} & = \sum_kX_k^{n+1}S_{ka}^{n+1}
=\sum_iS^n_{ia}\sum_kX_k^{n+1}\langle X_k^{n+1},X_i^n\rangle_x+ \nonumber\\
& \sum_kX_k^{n+1}\big(X_k^{n+1}U_a,\mathcal I_\theta(f^n)\big)_{xv}
=K_a^n+P_{x,n+1}\big((U_a,\mathcal I_\theta(f^n))_v\big).\label{eq:fixed-mode-update}
\end{align}}
{ 
Here $P_{x,n+1}$ is the orthogonal projection onto the spatial basis $\{X_k^{n+1}\}$, defined in \eqref{proj_x}.
Since $V_1 = U_1 = \frac1{\|1\|_v}$, the mass density is $\rho=\|1\|_v\sum_kX_kS_{k1}$ , and \eqref{eq:fixed-mode-update} with $a=1$ gives
$\rho^{n+1}=\rho^n+\|1\|_vP_{x,n+1}\big((U_1,\mathcal I_\theta(f^n))_v\big)$.
 Since $U_1=1/\|1\|_v$ is constant, $\|1\|_v\,(U_1,g)_v=\int g\,dv$ for every $g$. For the deterministic part, using that $E^n$ does not depend on $v$,
\begin{equation*}
\int\mathcal D(f^n)\,dv=-\nabla_x\cdot\int v\,f^n\,dv-E^n\cdot\int\nabla_vf^n\,dv=-\nabla_x\cdot J^n .
\end{equation*}
For the stochastic part of the Euler--Maruyama scheme, since $\sigma_j$ and $\sigma\sigma^\top$ depend only on $x$,
\begin{equation*}
\int\mathcal S_{\mathrm{EM}}(f^n)\,dv
=-\sum_{j=1}^K\Delta\beta^j_n\,\sigma_j(x)\cdot\int\nabla_vf^n\,dv
+\frac{\tau}{2}\int\nabla_v\cdot\big(\sigma\sigma^\top(x)\nabla_vf^n\big)\,dv=0 .
\end{equation*}
For the Heun scheme, inserting the predictor \eqref{heun_pred} gives
\begin{equation*}
\mathcal S_{\mathrm{Heun}}(f^n)=-\sum_{s=1}^K\Delta\beta^s_n\,\sigma_s(x)\cdot\nabla_vf^n
+\frac12\sum_{s,s'=1}^K\Delta\beta^s_n\Delta\beta^{s'}_n\,\nabla_v\cdot\Big(\sigma_s(x)\,\big(\sigma_{s'}(x)\cdot\nabla_vf^n\big)\Big),
\end{equation*}
where we used $\sigma_s\cdot\nabla_v g=\nabla_v\cdot(\sigma_s g)$, and therefore
\begin{align*}
\int\mathcal S_{\mathrm{Heun}}(f^n)\,dv
&=-\sum_{s=1}^K\Delta\beta^s_n\,\sigma_s(x)\cdot\int\nabla_vf^n\,dv\\
&\quad+\frac12\sum_{s,s'=1}^K\Delta\beta^s_n\Delta\beta^{s'}_n\int\nabla_v\cdot\Big(\sigma_s(x)\,\big(\sigma_{s'}(x)\cdot\nabla_vf^n\big)\Big)\,dv=0 .
\end{align*}
All the velocity integrals above vanish because their integrands are $v$-divergences; see the assumptions at the beginning of this section. Hence only the transport term contributes and
\begin{equation}\label{eq:mass-moment}
\|1\|_v\big(U_1,\mathcal I_\theta(f^n)\big)_v=-\tau\nabla_x\cdot J^n ,
\qquad\text{so that}\qquad
\rho^{n+1}=\rho^n-\tau P_{x,n+1}\big(\nabla_x\cdot J^n\big).
\end{equation}
The projection is exact: by the $K$-step for $k=1$ and \eqref{eq:mass-moment},
$K_1^{n+1}=K_1^n-\tau\nabla_x\cdot J^n/\|1\|_v$, and both $K_1^{n+1}$ and $K_1^n=\sum_iX_i^nS_{i1}^n$
belong to $\operatorname{span}\{X_k^{n+1}\}$, hence so does $\nabla_x\cdot J^n$. This proves \eqref{cont_d_rho}.
Finally, integrating \eqref{cont_d_rho} over the spatial domain and using the periodic boundary condition,
$\int \nabla_x\cdot J^n\,d x=0$, gives conservation of total mass. }
\end{proof}

\subsection{Local Momentum Conservation Law}

In this subsection we prove the discrete local momentum law for the two augmented BUG
schemes. The Euler--Maruyama-based BUG integrator and the Heun-based BUG integrator use
the same conservative truncation strategy {and the stochastic increments enter the
conservation proof in the same way: the second-order term (the It\^o correction for
Euler--Maruyama, the quadratic term generated by the predictor for Heun) has zero first
velocity moment. The next lemma shows that the standard augmentation from Section \ref{aug_trunc} is
sufficient for both integrators.}

\begin{lemma}\label{lem:Erho}
Consider the Euler--Maruyama-based {or the Heun-based} BUG integrator. For each component
$i=1,\ldots,d$,
\[
{-\tau\big(\nabla_x\cdot\Pi^n\big)_i+}\left(\tau E^n+\sum_{s=1}^K\Delta\beta_n^s\sigma_s(x)\right)_i\rho^n
\in \operatorname{span}\{X_k^{n+1}\},
\]
{where $\Pi^n(x)=\int_{\mathbb R^{\text d}}v\otimes v\,f^n(x,v)\,dv$.}
\end{lemma}

\begin{proof}
{Take $k=i+1$ in \eqref{eq:K-coeff}.}
For the momentum mode $k=i+1$, the It\^o correction term has
zero contribution ($c^3_{ln'jk}=0$), since the second velocity derivative has zero first velocity moment.
Using the fixed velocity mode $U_{i+1}={v_i/\|v_i\|_{v}}$ and integration by parts in
the velocity variable,
\begin{align*}
\sum_j c^2_{i+1,j}K_j^n
&=
\frac{1}{{\|v_i\|_{v}}}\sum_j
\int v_i\nabla_v(f_{0v}K_j^nV_j^n)\,dv
=
-\frac{e_i}{{\|v_i\|_{v}}}
\int f_{0v}\sum_jK_j^nV_j^n\,dv\\
&=-\frac{e_i}{{\|v_i\|_{v}}}\rho^n ,
\end{align*}
where $e_i$ is the $i$-th Euclidean unit vector. {The same computation for a single $V_j^n$ gives
$c^2_{i+1,j}=-\frac{\|1\|_v}{\|v_i\|_v}\,e_i\,\delta_{j1}$,
i.e.\ the $(i+1)$-th row of $c^2$ has a single nonzero entry, on the constant mode. Moreover
$c^2_{1\ell}=\frac{1}{\|1\|_v}\int\nabla_v(f_{0v}V^n_\ell)\,dv=0$. Hence the quadratic Heun term \eqref{eq:K-heun-term}
also has zero contribution on the momentum mode:
$\sum_{j}(c^2_{i+1,j}\cdot\sigma_s)(c^2_{j\ell}\cdot\sigma_r)
=(c^2_{i+1,1}\cdot\sigma_s)(c^2_{1\ell}\cdot\sigma_r)=0$.
Finally, $\sum_jc^1_{i+1,j}K_j^n=\frac{1}{\|v_i\|_v}\int v_i\,v\,f^n\,dv$, so that
$\sum_jc^1_{i+1,j}\cdot\nabla_xK_j^n=\frac{1}{\|v_i\|_v}(\nabla_x\cdot\Pi^n)_i$.
Therefore, for both integrators,
\begin{align}
\|v_i\|_v\big(K_{i+1}^{n+1}-K_{i+1}^n\big)&=\|v_i\|_v\big(U_{i+1},\mathcal I_\theta(f^n)\big)_v\nonumber\\
&=-\tau\big(\nabla_x\cdot\Pi^n\big)_i+\left(\tau E^n+\sum_{s=1}^K\Delta\beta_n^s\sigma_s(x)\right)_i\rho^n .\label{eq:momentum-moment}
\end{align}
Since $K_{i+1}^{n+1}$ and $K_{i+1}^n$ lie in the augmented spatial space, so does the right-hand side.}
\end{proof}

\begin{theorem}\label{thm_2}
Assume that the fixed velocity space contains the constant mode and the momentum modes {and, for both integrators, the augmentation of Section \ref{aug_trunc}.}
 Then both schemes satisfy the pathwise discrete local
momentum law, for $n=0,\ldots,N-1$,
\begin{equation}
\label{cont_d_j}
J^{n+1}-J^n+\tau(\nabla_x\cdot\Pi^n\new{-}E^n\rho^n)
\new{-}\sum_{s=1}^K\sigma_s(x)\Delta\beta_n^s\rho^n=0,
\end{equation}
where $\Pi^n$ is defined in Lemma~\ref{lem:Erho}.
Consequently, both integrators preserve the total momentum in expectation:
$\mathbb E[\mathcal P[f^n]]=\mathbb E[\mathcal P[f^0]]$, $0\le n\le N$.
\end{theorem}

\begin{proof}
Since $J_i=\|v_i\|_v\sum_kX_kS_{k,i+1}$, \eqref{eq:fixed-mode-update} with $a=i+1$ and \eqref{eq:momentum-moment} give
\[
J_i^{n+1}=J_i^n+P_{x,n+1}\Big(-\tau\big(\nabla_x\cdot\Pi^n\big)_i+\Big(\tau E^n+\sum_{s=1}^K\Delta\beta_n^s\sigma_s(x)\Big)_i\rho^n\Big).
\]
By Lemma \ref{lem:Erho}, which holds for both $\theta=\mathrm{EM}$ and $\theta = \mathrm{Heun}$, the projection is exact, which proves \eqref{cont_d_j}. Integrating \eqref{cont_d_j} over
the spatial domain and taking expectation gives the stated conservation of total
momentum, since the Brownian increments have mean zero. 
\end{proof}

\begin{remark}
    The proofs of Theorems~\ref{thm_1}, \ref{thm_2} and~\ref{thm:unified-energy} use only two properties of the stochastic
increment: it is the same in the $K$- and $S$-steps, and it is a polynomial in the
operators $\sigma_s(x)\cdot\nabla_v$ whose coefficients depend on the Brownian
increments and on the stepsize $\tau$, but not on $(x,v)$. The same conservation properties therefore hold for
higher-order increments of this form. Since the
third velocity derivatives of $1$, $v$ and $|v|^2$ vanish, terms of order three and
higher do not enter the discrete balance laws. We do not pursue such schemes here:
the augmented BUG integrator with explicit Euler substeps is only first-order
accurate in time, so a more accurate stochastic increment is not expected to improve
the overall accuracy.
\end{remark}

\subsection{Local Energy Conservation Law}
 
In this subsection, we present the local energy balance for the proposed integrators.
The result separates the kinetic part of the balance from the defect caused by the time
discretization of the electric field.
Here \(e(t,x)=\frac12\int_{\mathbb R^{\text d}}|v|^2f(t,x,v)\,dv+\frac12|E(t,x)|^2\) denotes the local energy density.
The continuous problem admits the following local energy balance:
\begin{align*}
de + \nabla_x\cdot\bigl[Q+u(J+\partial_tE)]\, dt
&= \sum_{k=1}^{K}  J \cdot \sigma_k \, d\beta_k +
\frac12 \mathrm{Tr}(\sigma\sigma^{\top}) \rho \, dt,
\end{align*}
where $Q(t,x) = \frac12 \int_{\mathbb{R}^{\mathrm d}} v |v|^2 f(t,x,v) \, dv$. Here $\nabla_x\cdot(J+\partial_tE)=0$, by the continuity and Poisson equations. Since $E=-\nabla_xu$, the total energy flux is $Q+u(J+\partial_tE)$. The additional flux has zero integrated divergence on the periodic domain.  
We now give the corresponding discrete statement.
Let us define
\begin{equation}\label{eq:theta-correction}
 \Theta_{\mathrm{EM}}^n
 =
 \frac{\tau}{2}\operatorname{Tr}(\sigma\sigma^{\top}(x))\rho^n,
 \qquad
 \Theta_{\mathrm{Heun}}^n
 =
 \frac12\sum_{s,{s'}=1}^K
 \Delta\beta^s_n\Delta\beta^{{s'}}_n
 \bigl(\sigma_s(x)\cdot\sigma_{{s'}}(x)\bigr)\rho^n .
\end{equation}
 
\begin{lemma} \label{lem:EJ}
Consider the Euler--Maruyama-based {or the Heun-based} BUG integrator.
It holds that
\[
{-\tau\nabla_x\cdot Q^n+}\left( \tau E^n  + \sum_{s=1}^K \Delta \beta^s_n \sigma_s(x) \right)\cdot J^n {+ \Theta^n_{\theta}}
\in \operatorname{span}\{X^{n+1}_k\},
\]
{where $Q^n=\frac12\int_{\mathbb R^{\text d}}v|v|^2f^n\,dv$.}
\end{lemma}
\begin{proof}
{We evaluate the $K$-step \eqref{eq:K-coeff} on the quadratic mode $k=\mathrm d+2$, i.e., the update of
$K_{\mathrm d+2}=\sum_iX_iS_{i,\mathrm d+2}$, which carries the kinetic energy. Recall that
$U_{\mathrm d+2}=\chi_2/\|\chi_2\|_v$ with $\chi_2=|v|^2-\alpha_2$, and $f^n=f_{0v}\sum_jK^n_jV^n_j$.
Every term of \eqref{eq:K-coeff} with $k=\mathrm d+2$ is a velocity moment of $f^n$, which we compute by integration by parts in $v$,
using $\nabla_v\chi_2=2v$ and $\partial_{v_l}\partial_{v_{n'}}\chi_2=2\delta_{ln'}$.
 
\emph{(i) Transport term.} By definition of $Q^n$ and $J^n$,
\[
\sum_jc^1_{\mathrm d+2,j}K^n_j=\frac{1}{\|\chi_2\|_v}\int\chi_2\,v\,f^n\,dv=\frac{1}{\|\chi_2\|_v}\big(2Q^n-\alpha_2J^n\big).
\]
 
\emph{(ii) Force and noise terms.} Integrating by parts once,
\[
\sum_jc^2_{\mathrm d+2,j}K^n_j=\frac{1}{\|\chi_2\|_v}\int\chi_2\,\nabla_vf^n\,dv=-\frac{2}{\|\chi_2\|_v}\int v\,f^n\,dv=-\frac{2}{\|\chi_2\|_v}J^n ,
\]
so that the force and noise terms of \eqref{eq:K-coeff} give
$\frac{2}{\|\chi_2\|_v}\big(\tau E^n+\sum_s\Delta\beta^s_n\sigma_s(x)\big)\cdot J^n$.
The same integration by parts applied to a single basis function gives the entries of the $(\mathrm d+2)$-th row of $c^2$:
\[
c^2_{\mathrm d+2,j}=\frac{1}{\|\chi_2\|_v}\int\chi_2\,\nabla_v(f_{0v}V^n_j)\,dv
=-\frac{2}{\|\chi_2\|_v}\langle v,V^n_j\rangle_v
=-\frac{2}{\|\chi_2\|_v}\sum_{b=1}^{\mathrm d}\|v_b\|_v\,e_b\,\delta_{j,1+b},
\]
where the last equality follows from $v_b=\|v_b\|_vU_{1+b}$, $V^n_{1+b}=U_{1+b}$ and the orthonormality of $\{V^n_j\}$. Hence $c^2_{\mathrm d+2,1+b}=-\frac{2\|v_b\|_v}{\|\chi_2\|_v}\,e_b$ for $b=1,\dots,\mathrm d$, and $c^2_{\mathrm d+2,j}=0$ for all other $j$.
 
\emph{(iii) Second-order term, Euler--Maruyama.} Integrating by parts twice,
\[
\sum_jc^3_{\mathrm d+2,jln'}K^n_j=\frac{1}{\|\chi_2\|_v}\int\chi_2\,\partial_{v_l}\partial_{v_{n'}}f^n\,dv
=\frac{2\,\delta_{ln'}}{\|\chi_2\|_v}\,\rho^n ,
\]
so that the It\^o correction in \eqref{eq:K-coeff} gives
$\frac{\tau}{2}\sum_{l,n'}A_{ln'}\frac{2\delta_{ln'}}{\|\chi_2\|_v}\rho^n=\frac{2}{\|\chi_2\|_v}\Theta^n_{\mathrm{EM}}$.
 
\emph{(iv) Second-order term, Heun.} By (ii) and $c^2_{1+b,\ell}=-\frac{\|1\|_v}{\|v_b\|_v}\,e_b\,\delta_{\ell1}$ (proof of Lemma~\ref{lem:Erho}),
\[
\sum_j\big(c^2_{\mathrm d+2,j}\cdot\sigma_s\big)\big(c^2_{j\ell}\cdot\sigma_{s'}\big)
=\sum_{b=1}^{\mathrm d}\big(c^2_{\mathrm d+2,1+b}\cdot\sigma_s\big)\big(c^2_{1+b,\ell}\cdot\sigma_{s'}\big)
=\frac{2\|1\|_v}{\|\chi_2\|_v}\,(\sigma_s\cdot\sigma_{s'})\,\delta_{\ell1},
\]
and, since $\|1\|_vK^n_1=\rho^n$, the quadratic Heun term \eqref{eq:K-heun-term} gives $\frac{2}{\|\chi_2\|_v}\Theta^n_{\mathrm{Heun}}$.
 
Inserting (i)--(iv) in \eqref{eq:K-coeff} with $k=\mathrm d+2$ and multiplying by $\|\chi_2\|_v/2$, we obtain for both integrators
\begin{align}
\frac{\|\chi_2\|_v}{2}\big(K^{n+1}_{\mathrm d+2}-K^n_{\mathrm d+2}\big)
&=-\tau\nabla_x\cdot Q^n+\frac{\alpha_2\tau}{2}\nabla_x\cdot J^n\nonumber\\
&\quad+\Big(\tau E^n+\sum_{s=1}^K\Delta\beta^s_n\sigma_s(x)\Big)\cdot J^n+\Theta^n_\theta .\label{eq:energy-moment}
\end{align}
The left-hand side lies in $\operatorname{span}\{X^{n+1}_k\}$: $K^n_{\mathrm d+2}=\sum_iX^n_iS^n_{i,\mathrm d+2}\in\operatorname{span}\{X^n_i\}$, and $K^{n+1}_{\mathrm d+2}$ belongs to the augmented spatial space by construction, which contains $\operatorname{span}\{X^n_i\}$. The term $\nabla_x\cdot J^n$ lies in the same span by the proof of Theorem~\ref{thm_1}. Hence the remaining terms on the right-hand side of \eqref{eq:energy-moment}, which form the expression in the statement, lie in $\operatorname{span}\{X^{n+1}_k\}$.}
\end{proof}

In dimension $\text{d}$, the velocity basis has to contain the constant function, all coordinate
functions $v_1,\dots,v_\text{d}$, and the quadratic mode $\chi_2/\|\chi_2\|_v$. Thus one needs at
least $m\ge \text{d}+2$ velocity modes in order to represent mass, momentum, and kinetic energy. By
the definition of the low-rank factors in \eqref{low-rank-formula} and the basis functions $U_i$,
$\frac12\int |v|^2f^n\,dv=\frac12\|\chi_2\|_vK^n_{\text{d}+2}+\frac12\alpha_2\|1\|_vK^n_1$.
We define $e^n=\frac12\int |v|^2f^n\,dv+\frac12|E^n|^2$.

\begin{theorem}\label{thm:unified-energy}
Assume that the {fixed velocity space contains $1,v_1,\dots,v_\text{d},|v|^2$ and, for both
integrators, the augmentation of Section \ref{aug_trunc}.} Then both schemes satisfy the pathwise identity
\begin{equation}\label{eq:unified-energy-law}
 e^{n+1}-e^n
 +\tau\nabla_x\cdot Q^n
 -\sum_{s=1}^K\Delta\beta_n^s\sigma_s(x)\cdot J^n
 -\Theta_\theta^n
-\mathcal R_E^n
 =0,
\end{equation}
where the electric-field residual is
\begin{equation}\label{eq:electric-residual}
 \mathcal R_E^n
 =
 \frac12\bigl(|E^{n+1}|^2-|E^n|^2\bigr)+\tau E^n\cdot J^n
 =
 \tau E^n\cdot
 \left(\frac{E^{n+1}-E^n}{\tau}+J^n\right)
 +\frac12|E^{n+1}-E^n|^2 .
\end{equation}
In particular, for the Heun-based BUG integrator,
$\mathbb E\!\left[\Theta_{\mathrm{Heun}}^n\mid \mathcal F_{t_n}\right]
 =\frac{\tau}{2}\operatorname{Tr}(\sigma\sigma^{\top}(x))\rho^n$,
provided the Brownian increments are independent of $\mathcal F_{t_n}$ and satisfy
$\mathbb E[\Delta\beta_n^s\Delta\beta_n^r\mid\mathcal F_{t_n}]=\tau\delta_{sr}$. 
\end{theorem}

\begin{proof}
{By the representation of the kinetic energy and \eqref{eq:fixed-mode-update} with $a=1$ and $a=\text{d}+2$,
\begin{align*}
&\frac12\int |v|^2 f^{n+1}\,d v-\frac12\int |v|^2 f^n\,d v\\
&\qquad=P_{x,n+1}\Big(\frac{\|\chi_2\|_v}{2}\big(U_{d+2},\mathcal I_\theta(f^n)\big)_v+\frac{\alpha_2\|1\|_v}{2}\big(U_1,\mathcal I_\theta(f^n)\big)_v\Big).
\end{align*}
By \eqref{eq:energy-moment} and \eqref{eq:mass-moment}, the terms $\pm\frac{\alpha_2\tau}{2}\nabla_x\cdot J^n$ cancel, and the argument of the projection equals
$-\tau\nabla_x\cdot Q^n+\big(\tau E^n+\sum_s\Delta\beta_n^s\sigma_s(x)\big)\cdot J^n+\Theta^n_\theta$.
By Lemma~\ref{lem:EJ} the projection is exact, hence
\[
 \frac12\int |v|^2 f^{n+1}\,d v-\frac12\int |v|^2 f^n\,d v
 =
 -\tau\nabla_x\cdot Q^n
 +\tau E^n\cdot J^n
 +\sum_{s=1}^K\Delta\beta_n^s\sigma_s(x)\cdot J^n
 +\Theta_\theta^n .
\]}
Adding $\frac12(|E^{n+1}|^2-|E^n|^2)$ to both sides gives \eqref{eq:unified-energy-law}, with $\mathcal R_E^n$ as in \eqref{eq:electric-residual}.
Finally, for the Heun correction,
$
 \mathbb E[\Theta_{\mathrm{Heun}}^n\mid\mathcal F_n]
 =
 \frac12\sum_{s,r=1}^K
 \mathbb E[\Delta\beta_n^s\Delta\beta_n^r\mid\mathcal F_{t_n}]
 (\sigma_s\cdot\sigma_r)\rho^n
 =
 \frac{\tau}{2}\operatorname{Tr}(\sigma\sigma^{\top}(x))\rho^n .
$
\end{proof}
Note that
the identity \eqref{eq:unified-energy-law} is exact and does not require an a priori
continuity estimate for $E^{n+1}-E^n$.

\section{Numerical details and results}

\subsection{Implementation details}
\label{implementation}

The spatial derivative is approximated by centered finite differences with periodic boundary conditions. The velocity derivative is also approximated by centered finite differences on the truncated velocity interval, with the boundary treatment chosen consistently with the decay of \(f_{0v}\) and with the discrete integration-by-parts identities used below. To preserve the semi-discrete conservation properties, the discrete differential operators must reproduce the algebraic identities used in the analysis.
The Poisson equation is solved via FFT-based spectral inversion.

The discrete spatial derivative satisfies the periodic cancellation property
\[
\int \nabla_x u\,dx =0,
\]
ensuring the cancellation of divergence terms. Likewise, the discrete velocity derivative must preserve the integration-by-parts identity
\[
\int v\,\nabla_v(f_{0v}V_l)\,dv
=
-\int f_{0v}V_l\,dv.
\]
For this reason, the derivative of the product $(f_{0v}V_l)$ is computed directly rather than using the product rule, which is not exactly satisfied by centered finite differences. This is essential for the correct evaluation of the coefficients $c^2_{kj}$ and improves the energy behavior of the scheme.


To avoid the accumulation of round-off errors, coefficients that are analytically known are enforced exactly. In particular, the coefficients $c^2_{ak}$ associated with the conserved velocity modes are prescribed analytically, while the coefficients $c^3_{lnjk}$, which vanish identically for the constant and the momentum modes, are explicitly set to zero. Finally, although the fixed velocity modes are analytically orthonormal, they are re-orthonormalized before each update with respect to the discrete velocity inner product. All reported conservation errors are therefore computed after the conservative truncation step, using the final low-rank factors at the corresponding time level.

\subsection{Numerical experiments}
This part presents numerical experiments.  We consider the stochastic Vlasov--Poisson equation on the phase space $\mathbb{T} \times \mathbb{R}$:

\begin{equation}
\begin{cases}
\mathrm{d}f + \big( v \, \partial_x f + E(t,x) \, \partial_v f \big) \, \mathrm{d}t 
+ \sigma(x) \, \partial_v f \star \mathrm{d}\beta(t) = 0,\\[1em]
\partial_x E(t,x) = \displaystyle \int_{\mathbb{R}} f(t,x,v) \, \mathrm{d}v - 1,
\quad t \in [0,T], \ x \in \mathbb{T},
\end{cases}
\label{eq:vp_stochastic}
\end{equation}
where $\{\beta(t)\}_{t\in[0,T]}$ is a standard one-dimensional Brownian motion.
We consider two initial densities
\begin{equation}
\label{ts}
f_0(x,v) = \frac{1}{2\sqrt{2\pi}} \Big( e^{-(v-2.4)^2/2} + e^{-(v+2.4)^2/2} \Big) \Big( 1 + \alpha \cos(\frac{2 \pi}{L} x) \Big)
\end{equation}
and 
\begin{equation}
\label{ll}
f_0(x,v) = \frac{1}{\sqrt{2\pi}} e^{-v^2/2} \left( 1 + \alpha \cos\left(\frac{2\pi}{L}x\right) \right),
\end{equation}
with $\alpha >0$. 
The equation with initial density \eqref{ts} corresponds to the two-stream instability problem, while the density \eqref{ll} corresponds to linear Landau damping. These are two classical problems widely used to validate Vlasov--Poisson solvers.

\subsubsection{Conservation properties}
The experiments in this section confirm the conservation properties proved in the paper. We first consider the two-stream instability problem, given by equation \eqref{eq:vp_stochastic} with \eqref{ts}. For this experiment, we set $\sigma(x) = 0.1 \sin(0.4 x)$ and  $\alpha = 0.001$.

The spatial domain is $\mathbb{T} = [0, L]$ with $L = 10\pi$, equipped with periodic boundary conditions and discretized using $n_x = 128$ grid points. The velocity domain is truncated to $[-7, 7]$ with $n_v = 128$ grid points and periodic boundary conditions. 
We fix $m = 3$ basis functions.
For the time discretization, we test both the Euler--Maruyama and Heun schemes with time step $\tau = 10^{-3}$ and fixed rank $7$. We run both methods over a single path up to final time $T=15$ and observe excellent pathwise mass conservation, with errors at machine precision, as shown in the right panels of Figures~\ref{heun_ts} and \ref{em_ts}. 

Concerning momentum conservation, the initial momentum is $J_0 = -1.219582 \times 10^{-4}$, instead of zero, due to the velocity grid discretization. This small value is caused by the nearly symmetric velocity grid, for which the continuous momentum integrand is odd in \(v\).
Monte Carlo simulations are performed with $n_{\text{paths}} = 10^4$ independent realizations  up to final time $T = 1$.  {The reported drift is the difference between the empirical mean of the total momentum at \(T\) and the initial discrete momentum \(J_0\).}

The computed mean momentum drift is:

\begin{itemize}
\item for the {Euler--Maruyama} scheme: $3.9485 \times 10^{-11}$;
\item for the {Heun} scheme: $1.3289 \times 10^{-9}$.
\end{itemize}

Both methods therefore preserve the mean momentum with excellent accuracy; see Figures~\ref{heun_ts} and \ref{em_ts}.

\begin{figure}[htbp]
    \centering

    \begin{subfigure}{0.45\textwidth}
        \centering
        \includegraphics[width=\textwidth]{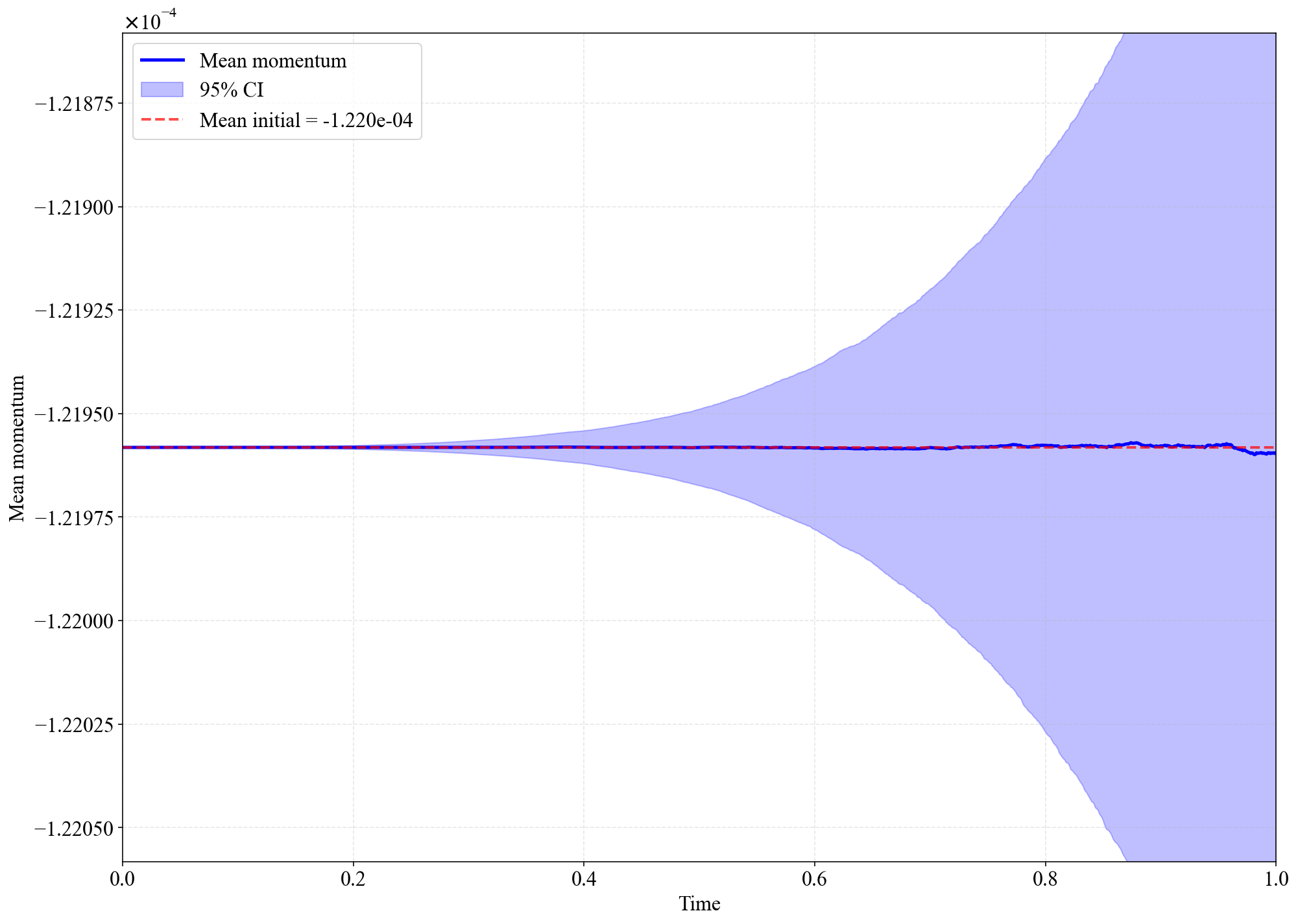}
        
    \end{subfigure}
    \hfill
    \begin{subfigure}{0.45\textwidth}
        \centering
        \includegraphics[width=\textwidth]{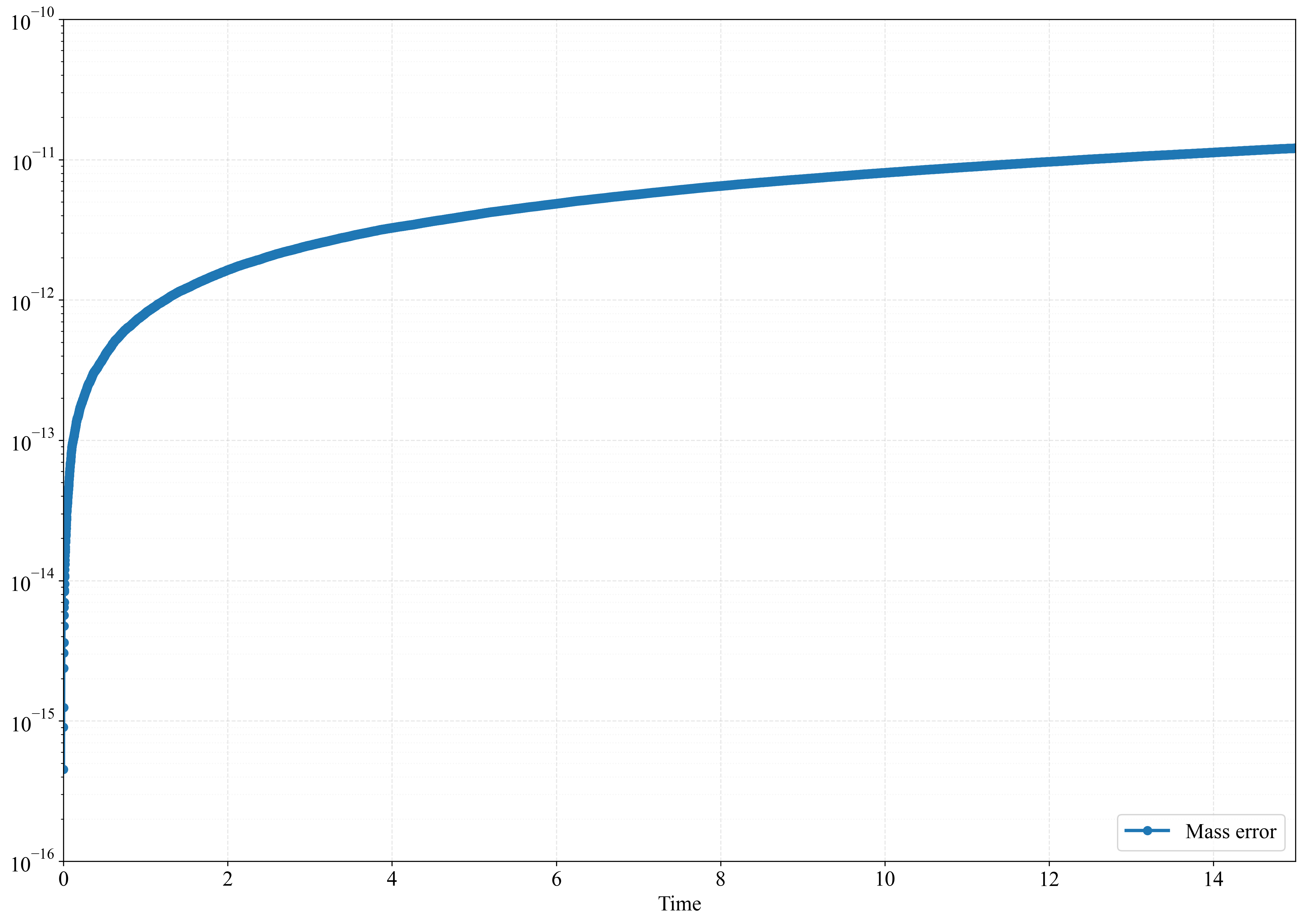}
       
    \end{subfigure}
\caption{Computational results for the Heun method on the two-stream instability problem. Mean momentum over $10^4$ paths (left). Mass error over a single path (right).}\label{heun_ts}
\end{figure}

\begin{figure}[h]
    \centering

    \begin{subfigure}{0.45\textwidth}
        \centering
        \includegraphics[width=\textwidth]{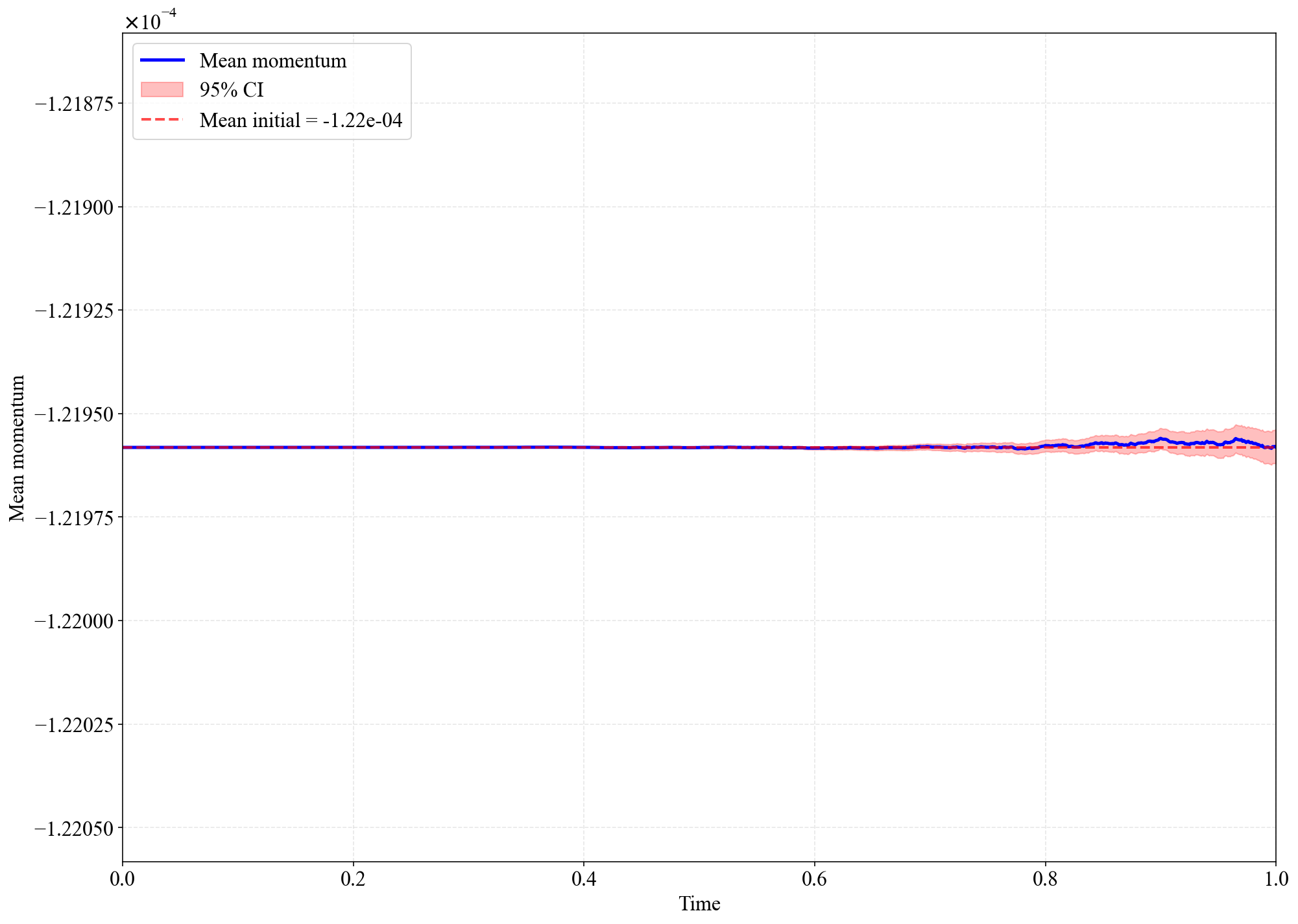}
    \end{subfigure}
    \hfill
    \begin{subfigure}{0.45\textwidth}
        \centering
        \includegraphics[width=\textwidth]{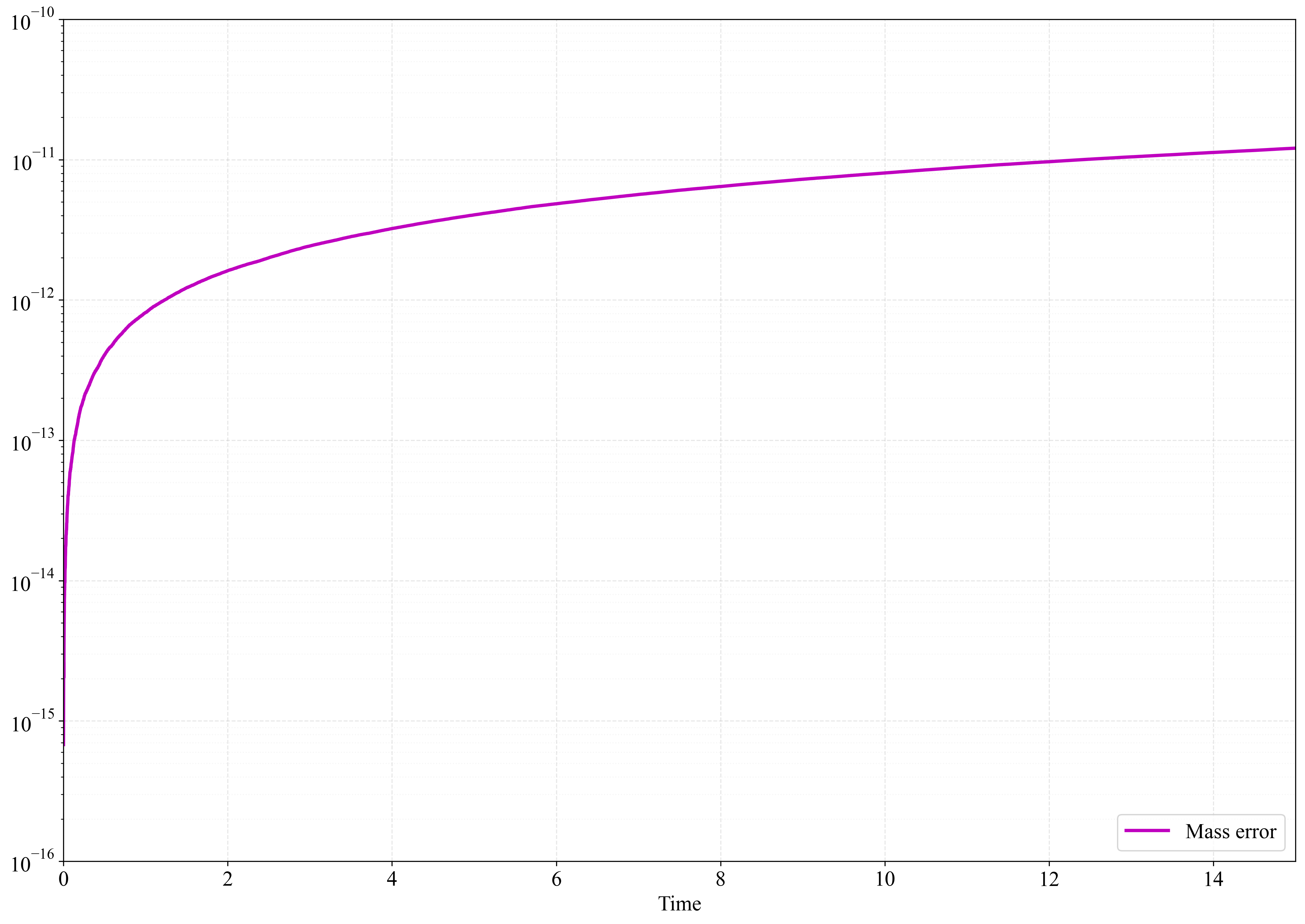}
    \end{subfigure}
\caption{Computational results for the Euler--Maruyama method on the two-stream instability problem. Mean momentum over $10^4$ paths (left). Mass error over a single path (right).}    \label{em_ts}
\end{figure}


In the second experiment, we consider the linear Landau problem, namely \eqref{eq:vp_stochastic} with initial condition \eqref{ll}, to test momentum conservation.
We set the spatial domain to $[0, L]$, with $L = 4\pi$ and $\alpha = 0.001$, and the velocity domain to $[-6, 6]$, with $\sigma(x) = 0.1 \cos(2x)$. The spatial and velocity grids and the time step are the same as in the previous experiment.
The rank is fixed at $5$. The number of independent paths in the Monte Carlo simulation is $10^4$. Figure~\ref{mom_ll} shows conservation in expectation for the Euler--Maruyama method (left) and the Heun method (right). The initial momentum is $J_0 = -4.294788 \times 10^{-8}$.  The final mean momentum drift computed at time $T=1$ is 
\begin{itemize}
    \item for the Euler--Maruyama scheme: $7.07326 \times 10^{-9} $;
    \item for the Heun scheme : $ 1.0999 \times 10^{-9}$.
\end{itemize}
These values are small relative to the scale of the discrete momentum and indicate that momentum conservation in expectation is achieved to high accuracy.


\begin{figure}[htbp]
    \centering
    
    \begin{minipage}{0.45\textwidth}
        \centering
        \includegraphics[width=\textwidth]{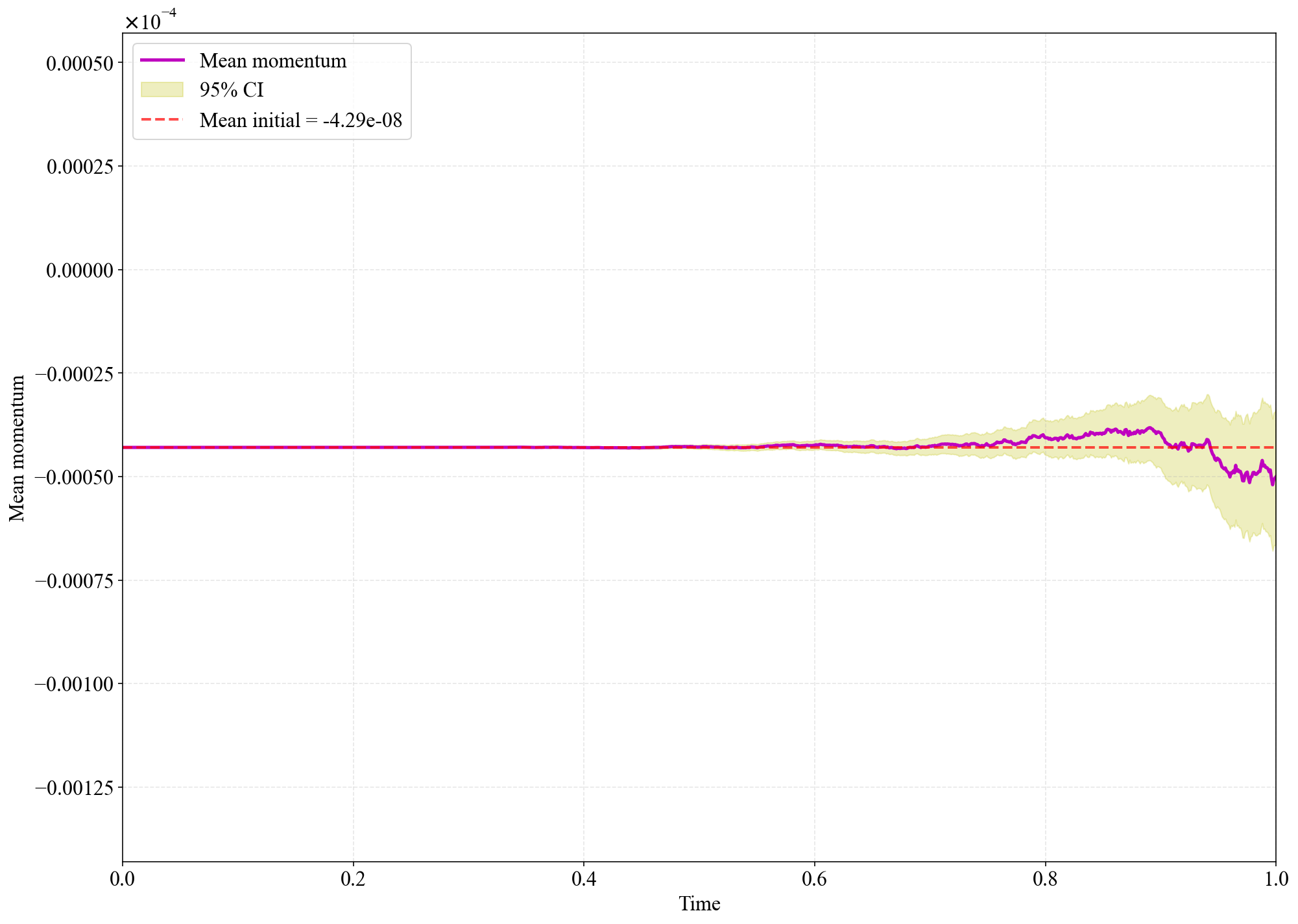}
    \end{minipage}
    \hfill
    \begin{minipage}{0.45\textwidth}
        \centering
        \includegraphics[width=\textwidth]{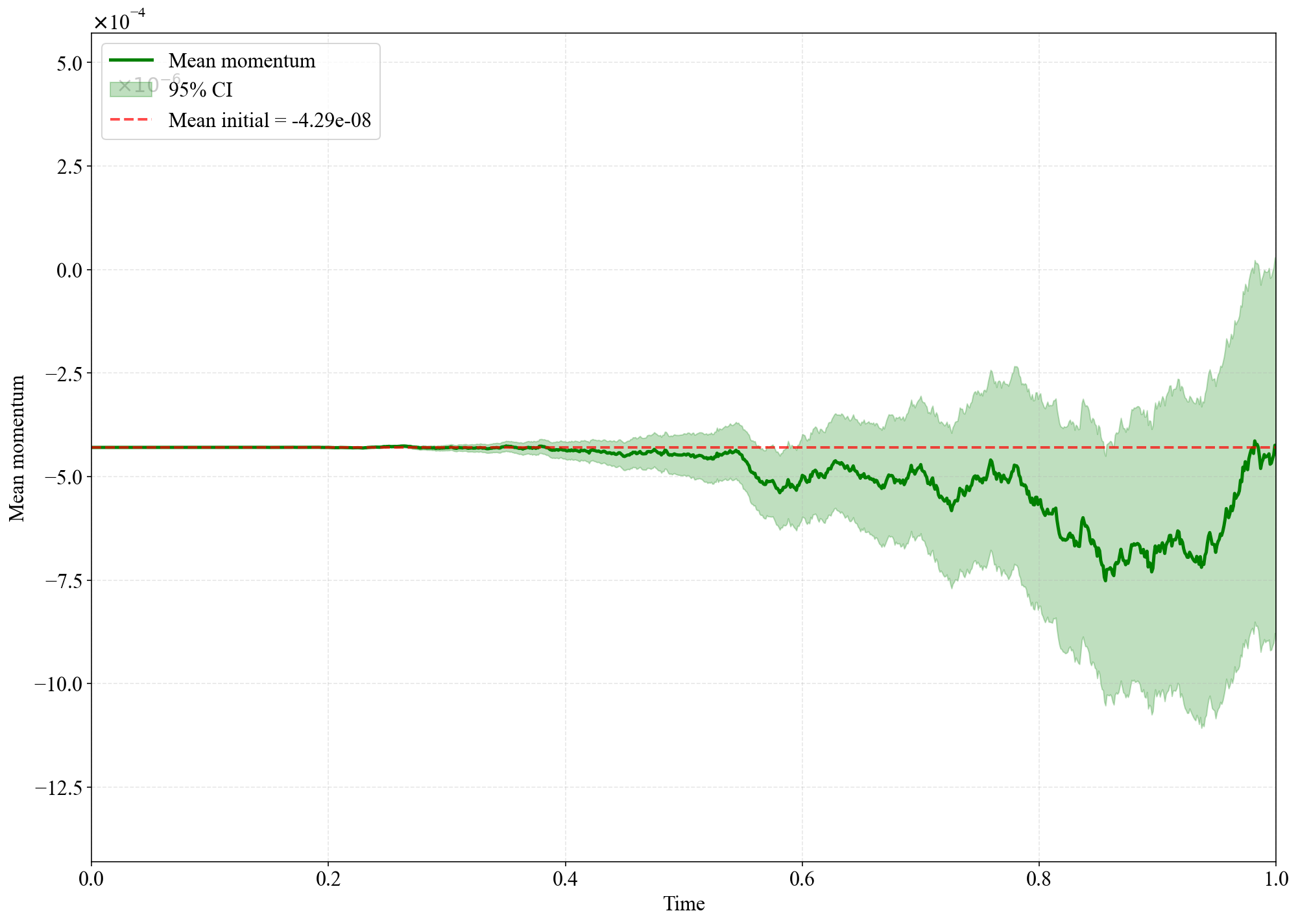}
    \end{minipage}
    
    \caption{Mean momentum for the linear Landau test case over $10^4$ Monte Carlo paths. Euler--Maruyama (left) and Heun (right).}
    \label{mom_ll}
\end{figure}


As a third experiment, we return to the two-stream instability problem with the same parameters as above, except that the noise intensity is constant, $\sigma = 0.1$. We test the Euler--Maruyama method. In particular, we run   a single path with fixed rank $15$, to ensure greater accuracy in the long-term evolution of the electric energy, and final time $T=25$. We observe in the top panels of Figure~\ref{ts_const_em} excellent mass conservation and near-perfect reproduction of the electric energy, analogous to the deterministic case and in agreement with the theory. 
Monte Carlo simulations are performed with $n_{\text{paths}} = 10^5$ independent realizations up to final time $T = 1$ and rank $5$.  The bottom panel of Figure~\ref{ts_const_em} shows the simulation results.
The ensemble-averaged momentum is found to be consistent with zero, as the final drift ($2.81\times 10^{-3}$) lies well within the $95\%$ confidence interval ($\pm 1.30\times 10^{-2}$). This indicates that, although individual paths exhibit large fluctuations ($\sigma_{\text{path}} \approx 2.09$), the DLR method preserves momentum in a statistical sense, with no evidence of systematic violation of the conservation law.
\begin{figure}[htbp]
    \centering
    \begin{subfigure}{0.45\textwidth}
        \centering
        \includegraphics[width=\textwidth]{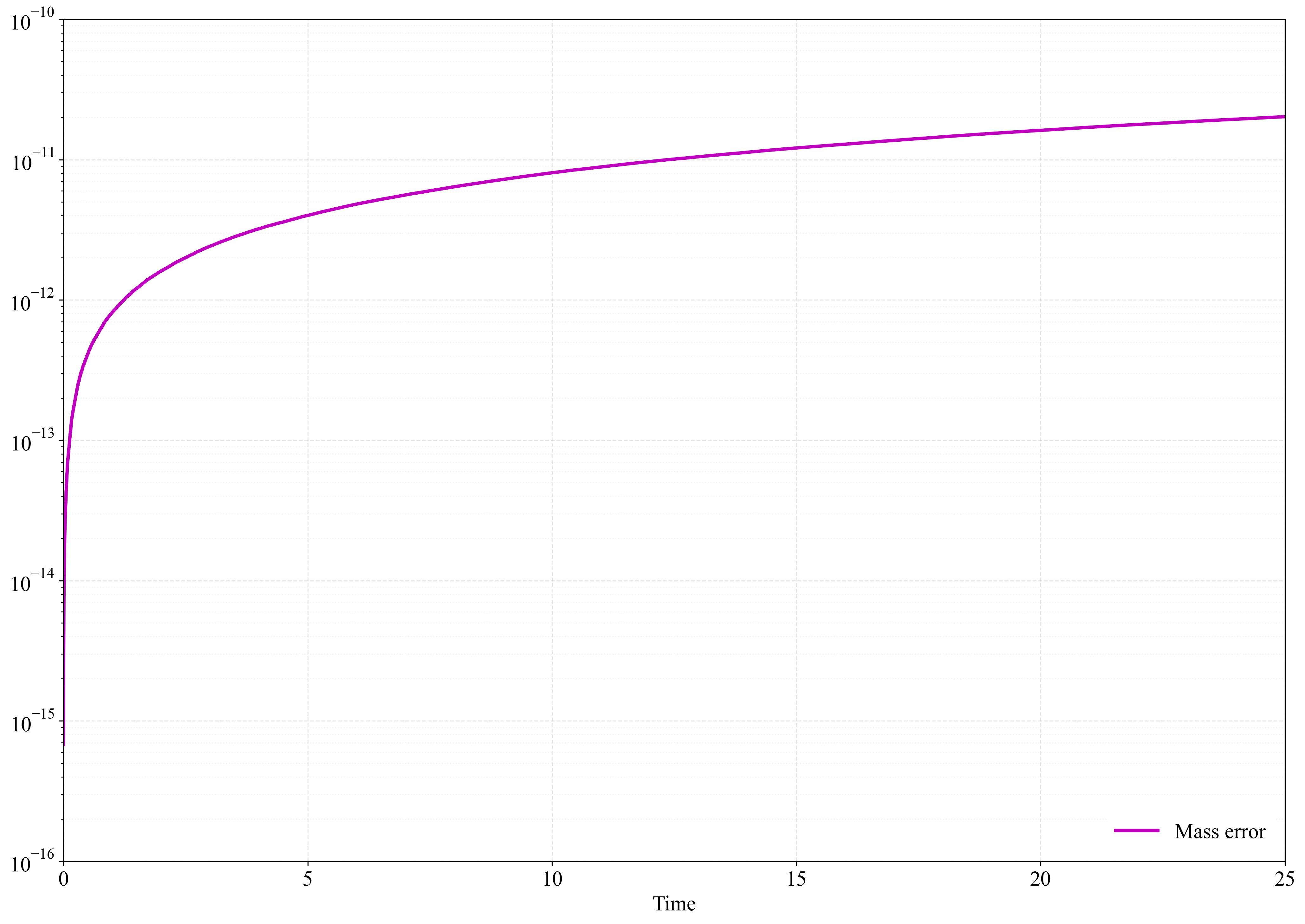}
    \end{subfigure}
    \hfill
    \begin{subfigure}{0.45\textwidth}
        \centering
        \includegraphics[width=\textwidth]{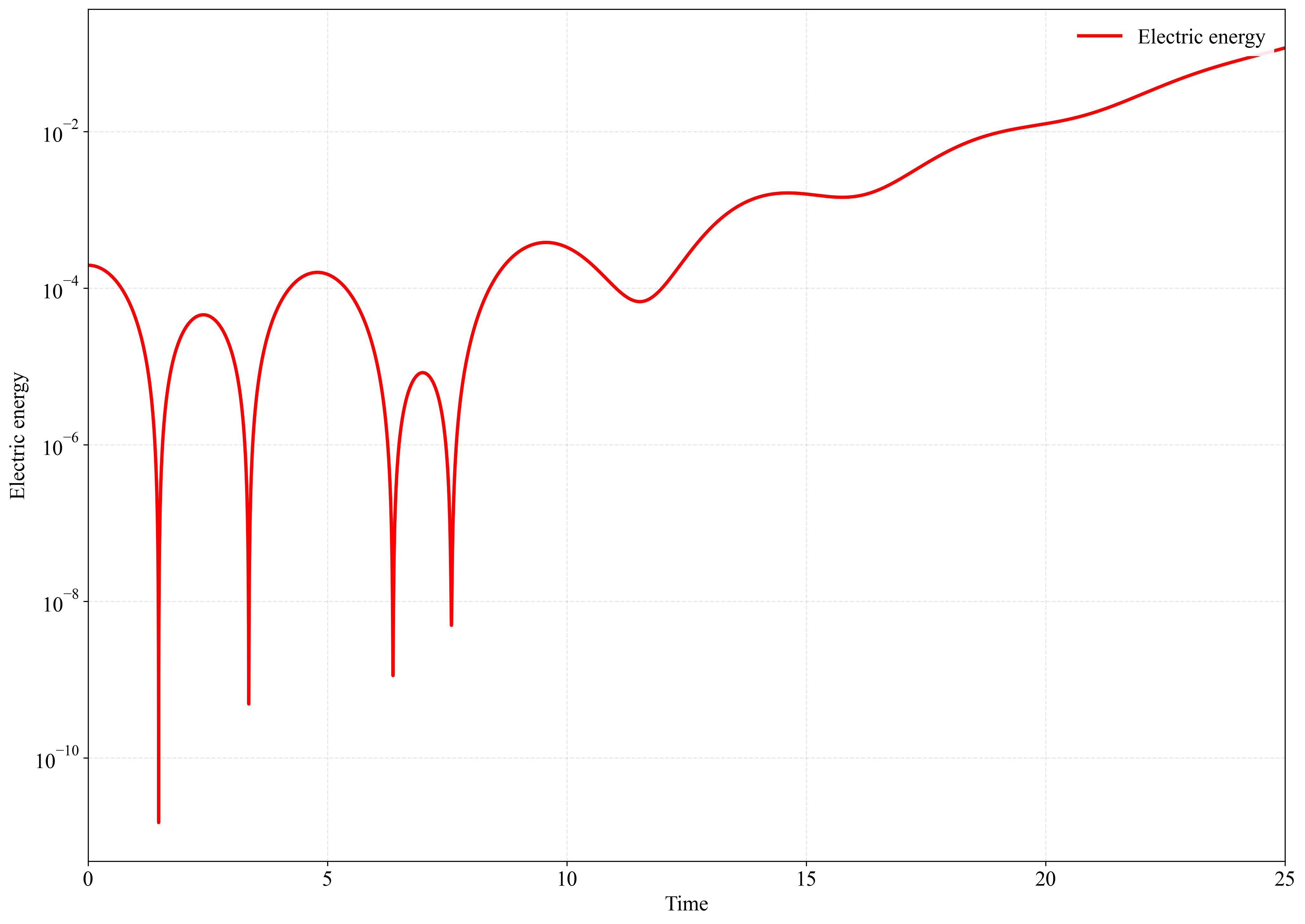}
    \end{subfigure}
    \begin{subfigure}{0.45\textwidth}
        \centering
        \includegraphics[width=\textwidth]{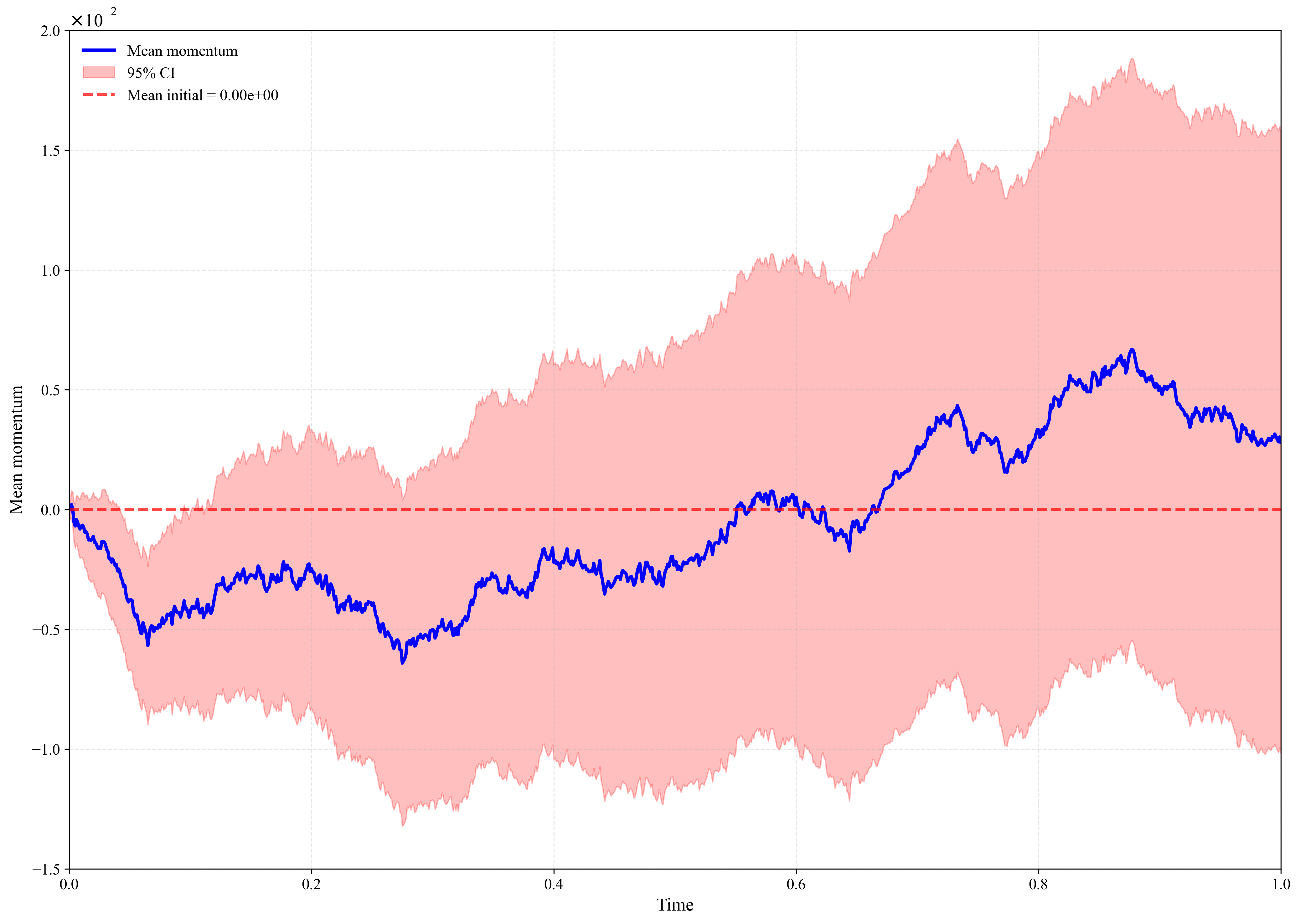}
    \end{subfigure}
    \caption{Euler--Maruyama method tested on the two-stream instability problem with noise intensity $\sigma = 0.1$. Top left: mass conservation over a single path. Top right: reproduction of the electric energy. Bottom: mean momentum over $10^5$ paths and rank $5$.}\label{ts_const_em}
    \end{figure}


\subsubsection{\texorpdfstring{Rank adaptation and the case \(m=0\)}{Rank adaptation and the case m=0}}
In this subsection, we consider the linear Landau test problem \eqref{eq:vp_stochastic} with \eqref{ll}, using the same parameter choices as above. We apply the Heun method with adaptive rank truncation; see Section~\ref{aug_trunc}. Figure~\ref{heun_ada} shows the mass conservation error, which remains at machine precision pathwise, and the rank adaptation over time for tolerance $\theta = 10^{-4}$.

  \begin{figure}[htbp]
    \centering

    \begin{subfigure}[b]{0.45\textwidth}
        \centering
        \includegraphics[width=\textwidth]{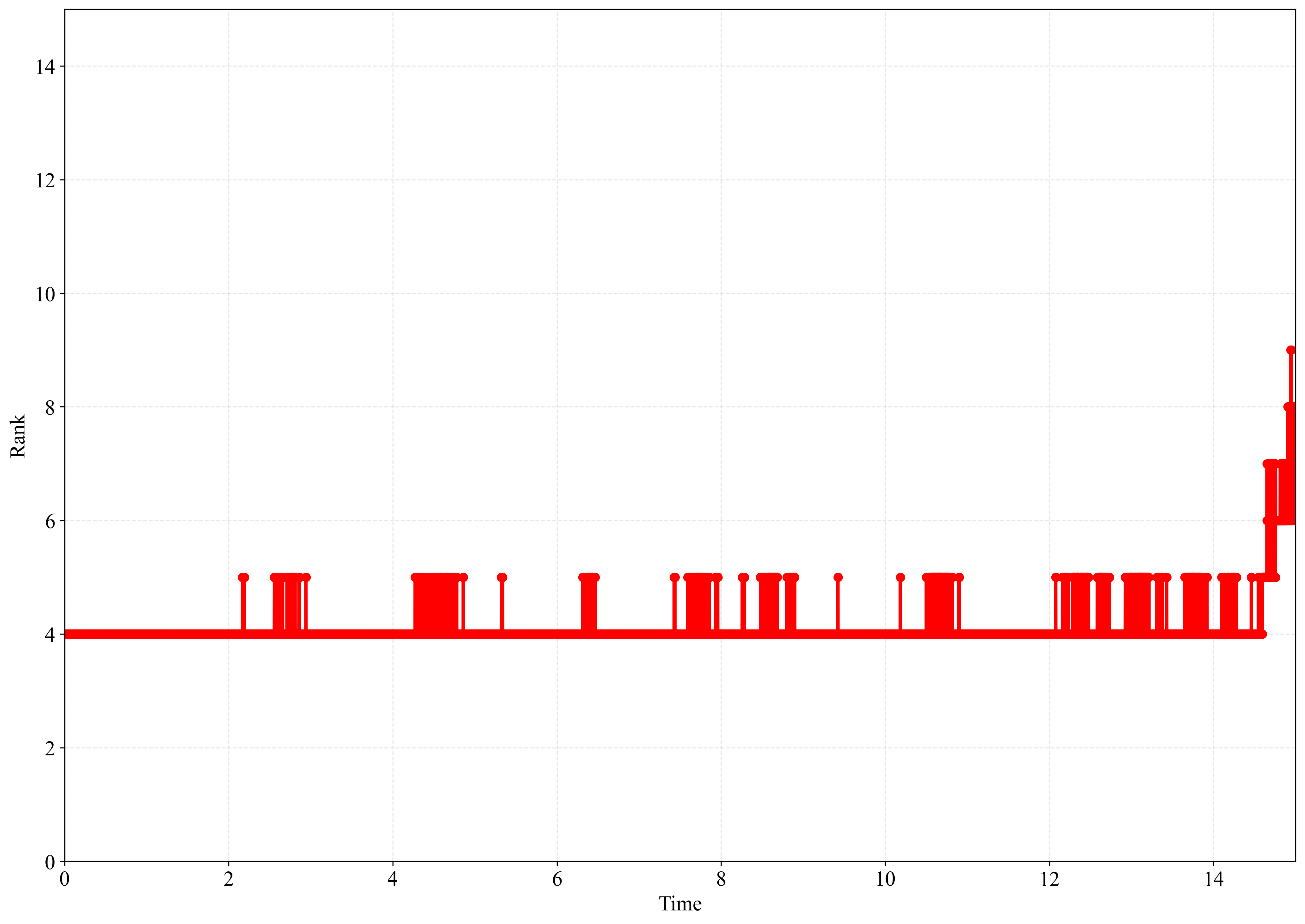}
        
    \end{subfigure}
    \hfill
    \begin{subfigure}[b]{0.45\textwidth}
        \centering
        \includegraphics[width=\textwidth]{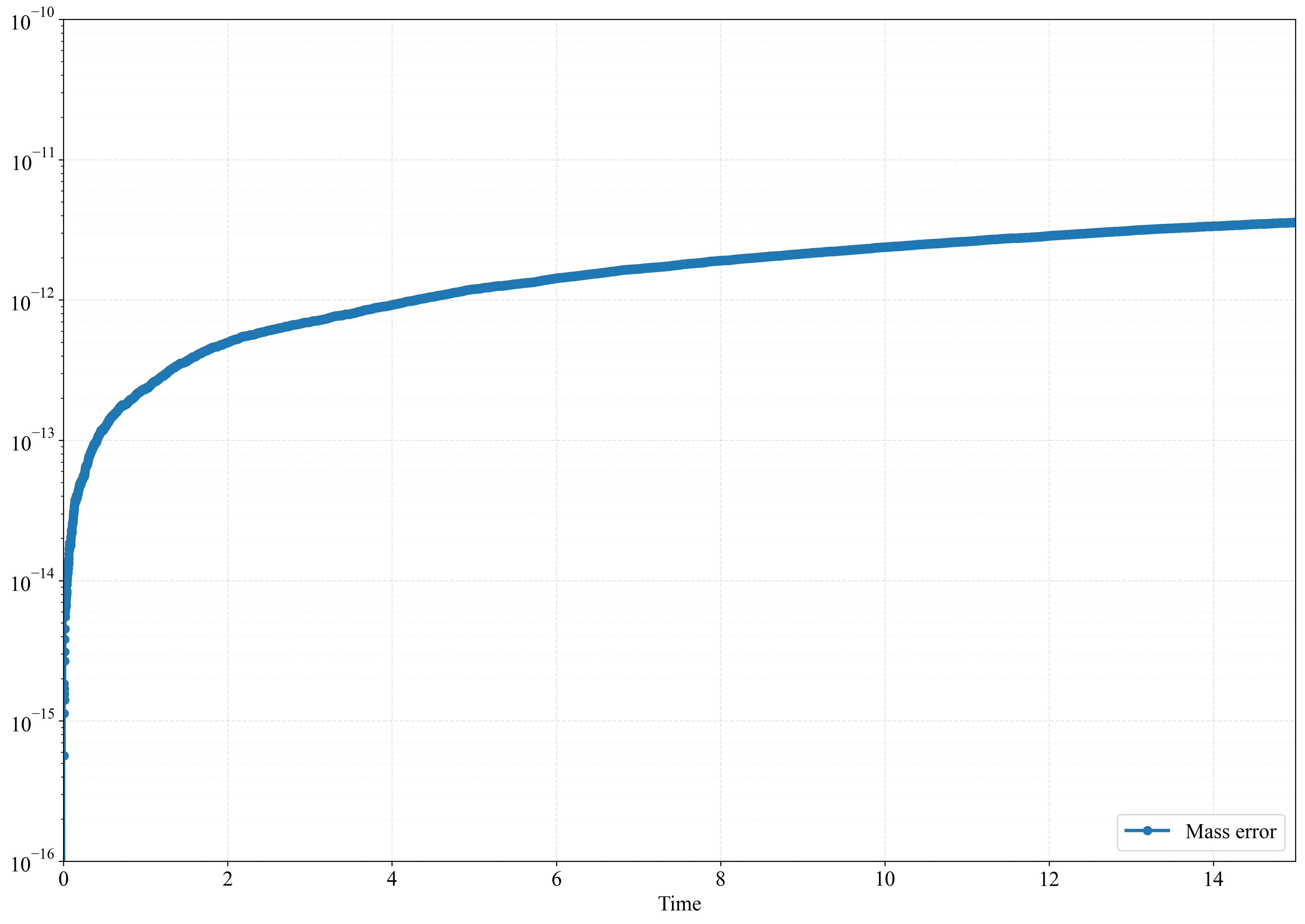}
        
    \end{subfigure}
    \caption{Heun method: rank adaptation over a single path (left) and mass error (right).}
    \label{heun_ada}
    \end{figure}
    
      \begin{figure}[htbp]
    \centering
    
\begin{subfigure}[b]{0.45\textwidth}
        \centering
        \includegraphics[width=\textwidth]{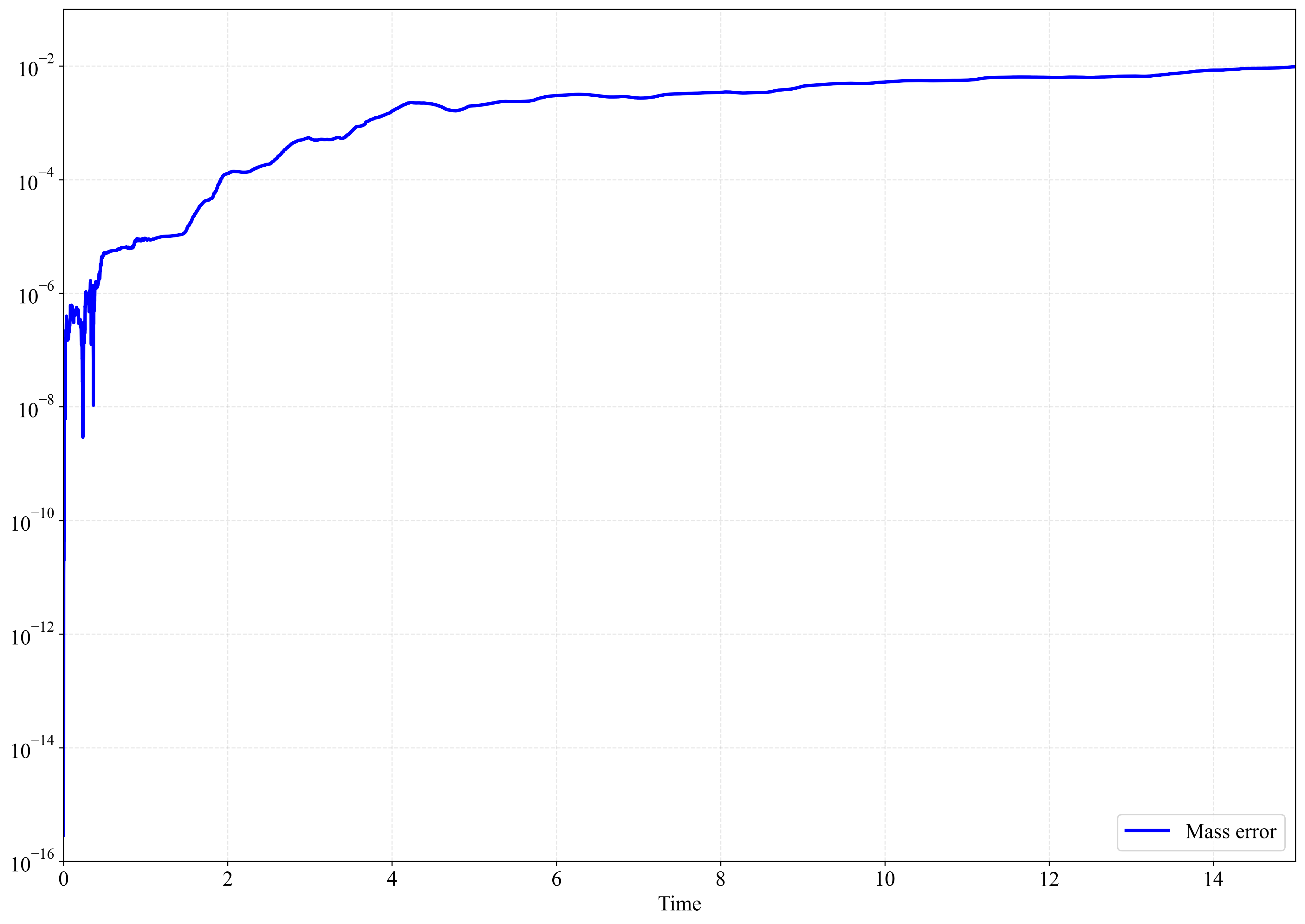}
    \end{subfigure}

    \caption{Mass error for the Euler--Maruyama method in the case $m=0$.}
    \label{fig:em}
\end{figure}
To demonstrate the crucial role of the fixed basis functions, we test Euler--Maruyama with rank $15$ without imposing mass conservation, i.e., with \(m = 0\). As shown in Figure~\ref{fig:em}, the mass exhibits a systematic drift, in contrast to the machine-precision conservation obtained with \(m = 3\). This confirms that the fixed basis functions are essential for enforcing the correct physical invariants in the low-rank approximation. 

\subsubsection{Strong convergence test} 
To assess the accuracy and robustness of the dynamical low-rank approximation in the stochastic setting, we perform a strong convergence test for both the Heun and Euler--Maruyama integrators. The test is based on the two-stream instability problem \eqref{eq:vp_stochastic} with \eqref{ts}. We set $\sigma(x) = 0.1 \sin(0.4 x)$ and  $\alpha = 0.001$.

The spatial domain is $\mathbb{T} = [0, L]$ with $L = 10\pi$, with periodic boundary conditions, discretized using $n_x = 128$ grid points. The velocity domain is truncated to $[-7, 7]$ with $n_v = 128$ grid points and periodic boundary conditions. 
We fix $m = 3$ basis functions.
A highly accurate reference solution is computed after finite-difference semi-discretization using a full-grid midpoint scheme with a fine time step $\tau_{\mathrm{ref}} = 10^{-6}$, which serves as a proxy for the exact solution. The relative error is measured in the $L^2$ norm at the final time $T = 0.5$ and averaged over $N_s = 15$ independent realizations.
The convergence test is performed for the time steps $\tau \in \{0.002,\; 0.001,\; 0.0005,\; 0.00025\}$. For each seed, a single Brownian path is generated at the reference resolution, and the coarse increments for each $\tau$ are obtained by aggregating the corresponding fine increments. This ensures that all time steps share the same noise realization. The same random increments are used when comparing Euler--Maruyama and Heun in the strong convergence test.

Figure~\ref{fig:convergence} shows the relative $L^2$ error as a function of the time step $\tau$ for the DLR-Heun and DLR-Euler--Maruyama methods, respectively. Both methods exhibit similar convergence behavior, with an average slope close to $0.5$.

\begin{figure}[htbp]
    \centering
    \begin{subfigure}[b]{0.45\textwidth}
        \centering
        \includegraphics[width=\textwidth]{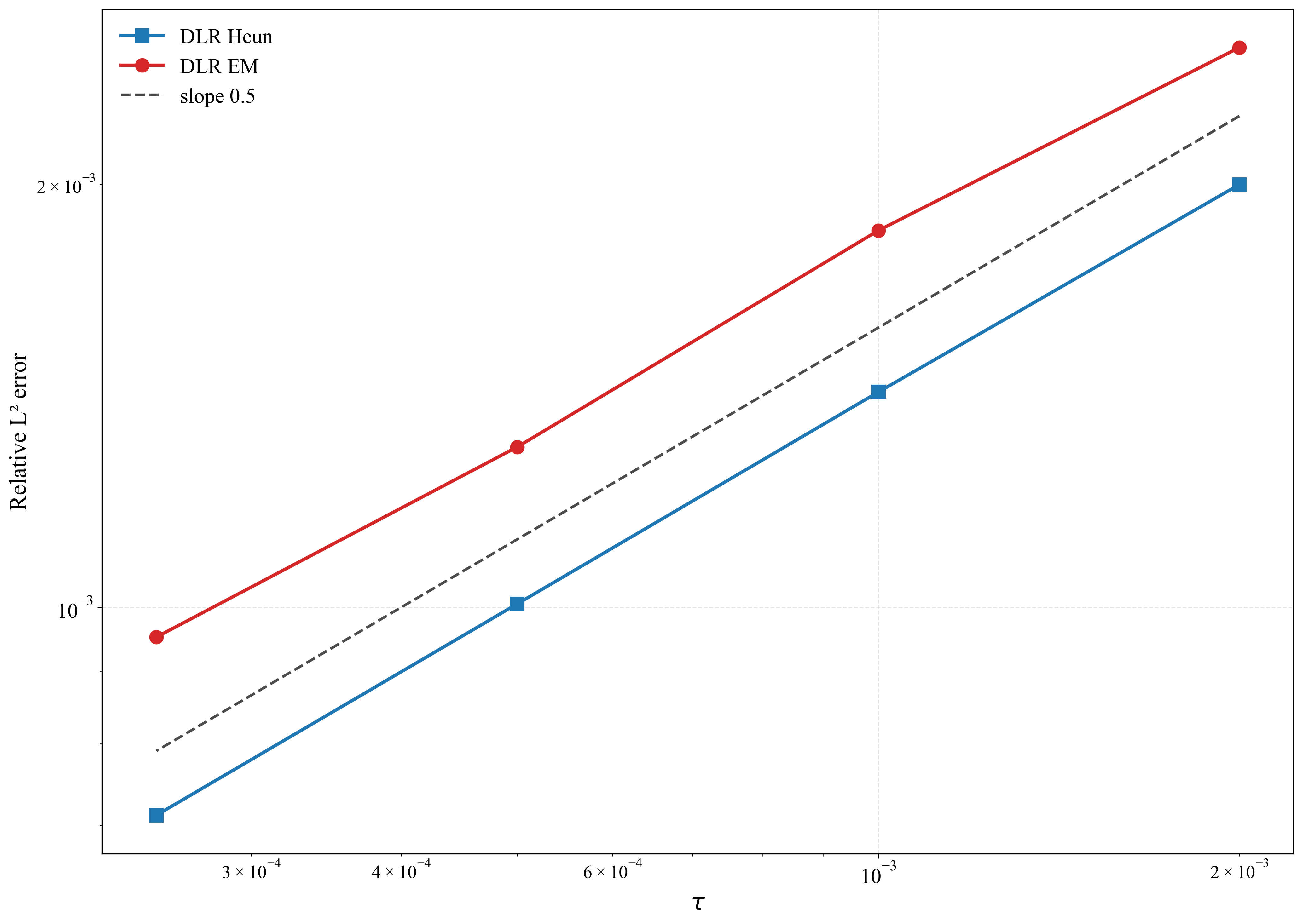}
      
    \end{subfigure}
     \caption{Strong convergence test for the DLR method applied to the two-stream instability. The reference solution is computed with \(\tau_{\mathrm{ref}}=10^{-6}\), and the relative \(L^2\) error is averaged over \(N_s=15\) paths.  }
    \label{fig:convergence}
\end{figure}

\section{Conclusions}
In this work, we have derived a conservative dynamical low-rank approximation that preserves the physical structure of the stochastic Vlasov--Poisson system with transport noise, extending the ideas in \cite{EJ2021}. After introducing a semi-discretization in space, we proposed an augmented conservative BUG integrator in the spirit of \cite{EOS}. The corresponding stochastic differential equations are integrated either with the Euler--Maruyama scheme for the equivalent It\^o formulation \eqref{sto-Vlasov-ito}, or with the Heun scheme, which is directly consistent with the Stratonovich formulation. The analysis {shows that the choice of stochastic discretization does not affect the discrete conservation properties and that, as in the deterministic case \cite{EOS, EKS},  standard augmentation of the spatial basis and conservative truncation are sufficient.} The numerical experiments confirm the conservation properties predicted by the theory. 
Large-scale multidimensional simulations, while of considerable practical interest, constitute a complementary computational topic and are left for future investigations; see, e.g., \cite{CE} for large-scale dynamical low-rank simulations of kinetic equations.
{Several questions remain open, including the convergence analysis, higher order methods and  the design of rank-adaptive strategies tailored to mean quantities.}
These topics will be pursued in future work.

\end{document}